%% file: main.tex
\documentclass{amsart}

\input{preamble.tex}

\title{On involutions of minuscule Kirillov algebras induced by real structures}
\author{Mischa Elkner}
\date{}

\begin{document}
\begin{abstract}
    We study Kirillov algebras attached to minuscule highest weight representations of semisimple Lie algebras. They can be viewed as equivariant cohomology algebras of partial flag varieties. Real structures on the varieties then induce involutions of these algebras. We describe how these involutions act on the spectra of minuscule Kirillov algebras, and model the fixed points via the equivariant cohomology of real partial flag varieties. We then use this model to characterise freeness of the fixed point coordinate ring over the appropriate base. As an application, we recover a $q=-1$ phenomenon of Stembridge in the minuscule case by geometric means.
\end{abstract}

\maketitle

\tableofcontents

\section{Introduction}
The finite-dimensional irreducible representations of a complex semisimple Lie algebra $\g$ are central to Lie theory. For such a representation $V^\lambda$, labelled by its highest weight $\lambda$,\footnote{With respect to some Borel and Cartan subalgebra. When this choice does not matter, we do not specify it, and simply speak of weights, roots, Weyl group etc.\,of $\g$.} the \emph{Kirillov algebra}
$$
    \Kir^\lambda(\g)
    \coloneqq
    (S(\g) \otimes \End(V^\lambda))^\g
    \vspace{3pt}
$$
was introduced in \cite{Kirillov2001_introduction}.\footnote{In \cite{Kirillov2001_introduction}, Kirillov algebras are called \emph{(classical) family algebras}.} Here, $S(\g) = \bigoplus_{n\ge 0}S^n(\g)$ is the symmetric algebra on the vector space $\g$, and $\End(V^\lambda)$ the algebra of linear endomorphisms of $V^\lambda$. Both have canonical $\g$-actions, and $\Kir^\lambda(\g)$ consists of the $\g$-invariants in their tensor product.

The Kirillov algebra is commutative if and only if the representation $V^\lambda$ is weight multiplicity free \cite[Cor.\,1]{Kirillov2001_introduction}. This is in particular the case when $\lambda$ is a minuscule weight, that is, one for which all weights of $V^\lambda$ lie in a single Weyl group orbit. Panyushev \cite{Panyushev2004_endoj} has observed that such minuscule Kirillov algebras admit geometric descriptions: they are isomorphic to equivariant cohomology rings of partial flag varieties for the connected simply connected Lie group $G$ with Lie algebra $\g$. Here we follow Hausel \cite{Hausel2024_avatars} in using a slightly different model, in terms of the Langlands dual group $G^\vee$ of $G$.\footnote{For our main purpose of studying minuscule Kirillov algebras, one could avoid Langlands duality and work only in terms of $G$. However, using $G^\vee$ is more natural in the context of general Kirillov algebras, see Theorem \ref{thm:geom_model} below.} Namely, the weight $\lambda$ defines a parabolic subgroup $P_\lambda$ of $G^\vee$, and we have a ring isomorphism
\begin{equation}
\label{eq:kir_alg_geom_model}
    \Kir^\lambda(\g) \cong H^{2*}_{G^\vee}(G^\vee/P_\lambda, \C).
\end{equation}

Let us mention two useful properties of this isomorphism. Firstly, the Kirillov algebra inherits a grading from $S(\g)$, and \eqref{eq:kir_alg_geom_model} becomes a graded isomorphism if the right-hand side is graded by half the cohomological degree. (Note that $G^\vee/P_\lambda$ has no odd cohomology, so this is merely a regrading of the full cohomology ring.) Secondly, $\Kir^\lambda(\g)$ is a graded algebra over the invariant ring $S(\g)^\g$, which embeds into $\Kir^\lambda(\g)$ as $f\mapsto f\otimes \id$. The isomorphism identifies this subring with $H^{2*}_{G^\vee}\coloneqq H^{2*}_{G^\vee}(\pt, \C)$, the equivariant cohomology of a point. Altogether, the geometric model is summarised by the following commutative diagram:
\begin{equation}
\label{eq:geom_model_diag}
\begin{tikzcd}
    \Kir^\lambda(\g)
        \arrow[r, "\cong"]
    & H^{2*}_{G^\vee}(G^\vee/P_\lambda,\C) \\
    S(\g)^\g
        \arrow[u, hook]
        \arrow[r, "\cong"]
    & H^{2*}_{G^\vee}.
        \arrow[u, hook]
\end{tikzcd}
\end{equation}

The goal of this paper is to describe automorphisms of minuscule Kirillov algebras induced by automorphisms of the $G^\vee$-space $G^\vee/P_\lambda$ through \eqref{eq:geom_model_diag}. More precisely, we focus on involutions arising in this way from real structures of $G^\vee$. Such a real structure (\ie\ an antiholomorphic automorphism) $\sigma$ of $G^\vee$ acts naturally on the coweights, and induces a real structure on $G^\vee/P_\lambda$ if $\sigma_*\lambda$ is in the same $G^\vee$-orbit as $\lambda$. Taking equivariant cohomology then yields an involution $\sigma^*$ of $H^{2*}_{G^\vee}(G^\vee/P_\lambda,\C) \cong \Kir^\lambda(\g)$ mapping the subring $H^{2*}_{G^\vee} \cong S(\g)^\g$ to itself. We are interested in the action of $\sigma^*$ on spectra. In particular, we want to describe the fixed point schemes, given by
\begin{equation}
\label{eq:fpschemes}
\begin{tikzcd}
    \Spec (\Kir^\lambda(\g)_{\sigma^*})
        \arrow[r, "\cong"]
        \arrow[d]
    &(\Spec \Kir^\lambda(\g))^{\sigma^*}
        \arrow[d]
        \arrow[r, hook]
    &\Spec \Kir^\lambda(\g)
        \arrow[d]
    \\
    \Spec (S(\g)^\g_{\sigma^*})
        \arrow[r, "\cong"]
    &(\Spec S(\g)^\g)^{\sigma^*}
        \arrow[r, hook]
    &\Spec S(\g)^\g.
\end{tikzcd}
\end{equation}
Here, the superscript $\sigma^*$ denotes sets of fixed points, which form closed affine subschemes whose coordinate rings $\Kir^\lambda(\g)_{\sigma^*}$ and $S(\g)^\g_{\sigma^*}$ we call \emph{coinvariant rings}. Dually to the left vertical map, we have a homomorphism 
\begin{equation}
\label{eq:coinv_alg_hom}
    S(\g)^\g_{\sigma^*}
    \to
    \Kir^\lambda(\g)_{\sigma^*}.
\end{equation}

To describe the homomorphism \eqref{eq:coinv_alg_hom} -- and with it, the fixed points of $\sigma^*$ -- a geometric model similar to \eqref{eq:geom_model_diag} is desirable. A simple candidate for such a model is to use the real form of $G^\vee/P_\lambda$ defined by $\sigma$. However, $\sigma^*$ depends only on the \emph{inner class} of $\sigma$, which essentially is to say that $\sigma^*$ is unaffected if we compose $\sigma$ with an inner automorphism of $\g$. This is not the case for the real form of $G^\vee/P_\lambda$, so one has to identify a suitable representative in a given inner class. It turns out that, roughly speaking, $\sigma$ should be chosen as compact as possible while preserving the parabolic $P_\lambda$. When made precise, being ``as compact as possible'' means that $\sigma$ should induce a \emph{quasi-compact} real structure (see Definition \ref{def:qcrs}) on a Levi subgroup of $P_\lambda$. We show that such a choice of representative $\sigma$ exists (see Section \ref{sec:main_proof}) and use this to obtain the desired model. Our main result is as follows:
\begin{thm}
\label{thm:main}
    Let $\g$ be a complex semisimple Lie algebra, $G$ the connected simply connected Lie group with Lie algebra $\g$, and $G^\vee$ the Langlands dual group of $G$. Let $\lambda$ be a minuscule weight of $G$, viewed as a cocharacter of $G^\vee$. Then any inner class $\mathfrak{S}$ of real structures of $G^\vee$ which fix the $G^\vee$-orbit of $\lambda$ contains a real structure $\sigma$ such that $\sigma_* \lambda = \lambda$ and
    \begin{equation}
    \label{eq:main_1}
        \Kir^\lambda(\g)_{\sigma^*}
        \cong
        H^{2*}_{(G^\vee)^\sigma}((G^\vee)^\sigma/P_\lambda^\sigma, \C)
    \end{equation}
    for the parabolic subgroup $P_\lambda\le G^\vee$ defined by $\lambda$. Moreover, the same inner class contains a quasi-compact real structure $\sigma_0$ with
    \begin{equation}
    \label{eq:main_2}
        S(\g)^\g_{\sigma^*}
        \cong
        H^{2*}_{(G^\vee)^{\sigma_0}},
    \end{equation}
    and there is a canonical injection
    \begin{equation}
    \label{eq:main_3}
        \varphi\colon
        H^{2*}_{(G^\vee)^{\sigma_0}}
        \hookrightarrow
        H^{2*}_{(G^\vee)^{\sigma}}.
    \end{equation}
    Combining \eqref{eq:main_1}-\eqref{eq:main_3} yields a commutative diagram
    \begin{equation}
    \label{eq:main_diag}
    \begin{tikzcd}
        \Kir^\lambda(\g)_{\sigma^*}
            \arrow[rr, "\cong"]
        &&
        H^{2*}_{(G^\vee)^\sigma}((G^\vee)^\sigma/P_\lambda^\sigma, \C)
        \\
        S(\g)^\g_{\sigma^*}
            \arrow[u]
            \arrow[r, "\cong"]
        & H^{2*}_{(G^\vee)^{\sigma_0}}
            \arrow[r, hook, "\varphi"]
        & H^{2*}_{(G^\vee)^\sigma}
            \arrow[u]
    \end{tikzcd}    
    \end{equation}
    in which the right vertical arrow is the structure map for equivariant cohomology.
\end{thm}
For a discussion of uniqueness of the real structure $\sigma$ in Theorem \ref{thm:main} see Appendix \ref{sec:uniqueness_and_tables}. 

Let us briefly sketch our approach and the role of quasi-compactness. Already on the level of the base ring $H^{2*}_{G^\vee}$, we can compare its coinvariant ring $(H^{2*}_{G^\vee})_{\sigma^*}$ to $H^{2*}_{(G^\vee)^\sigma} = H^{2*}_{(G^\vee)^\sigma}(\pt,\C)$. Chern--Weil theory (see Section \ref{sec:background}) identifies $H^{2*}_{G^\vee}$ with the invariant polynomial ring $\C[\g^\vee]^{\g^\vee}$ (where $\g^\vee$ is the Lie algebra of $G^\vee$). To similarly describe $H^{2*}_{(G^\vee)^\sigma}$, one associates to $\sigma$ its \emph{Cartan involution}, a complex linear Lie algebra involution $\theta$ of $\g^\vee$. Then $H^{2*}_{(G^\vee)^\sigma} \cong \C[(\g^\vee)^\theta]^{(\g^\vee)^\theta}$, and to ensure that this is the coinvariant ring of $\C[\g^\vee]^{\g^\vee}$ as desired, $\theta$ has to be particularly well-behaved. This translates precisely to the requirement that $\sigma$ be quasi-compact (see Section \ref{sec:invrings_and_qcrs}).

On the level of Kirillov algebras, a similar translation to invariant rings is possible. Here it turns out that $\sigma$ has to fix $\lambda$ (denoted $\sigma_*\lambda = \lambda$ in the theorem) and one then has to analyse the restriction of $\sigma$ to $L_\lambda$, the Levi factor of the parabolic $P_\lambda$. There is a further complication arising from the fact that $L_\lambda^\sigma$ need not be connected. Up to this issue, which we address, one arrives at the requirement that $\sigma|_{L_\lambda}$ be quasi-compact, analogously to the level of base rings. The main work in proving Theorem \ref{thm:main} is then to show that these requirements are fulfilled by some real structure $\sigma$ (with one exception, related to the issue of connectedness). Let us note that this $\sigma$ need not be quasi-compact on all of $G^\vee$, which leads to the auxiliary real structure $\sigma_0$ in the theorem, used on the level of base rings.

In addition to the well-behaved Cartan involutions there are also topological advantages to quasi-compactness. Indeed, if $\g$ is simple and not isomorphic to $\sl_{2n+1}$, then the real partial flag varieties $(G^\vee)^\sigma/P_\lambda^\sigma$ arising in Theorem \ref{thm:main} are \emph{equal rank homogeneous spaces} (see Section \ref{sec:ff}). Departing from this observation, we use Theorem \ref{thm:main} to describe a key algebraic property of the coinvariant homomorphism \eqref{eq:coinv_alg_hom}:
\begin{thm}
\label{thm:degeneracy}
    In the setting of Theorem \ref{thm:main}, either $\Kir^\lambda(\g)_{\sigma^*}$ is a free module of finite rank over $S(\g)^\g_{\sigma^*}$ or the homomorphism \eqref{eq:coinv_alg_hom} is non-injective. The former case holds precisely when we can choose $\sigma_0 = \sigma$ in Theorem \ref{thm:main}, which is equivalent to $\lambda$ being fixed by a quasi-compact real structure in the given inner class $\mathfrak{S}$. 
\end{thm}
Structural results like Theorem \ref{thm:degeneracy} are of inherent interest, but are also related to combinatorial applications, one of which we now describe. Namely, for a generic $x\in \Spec S(\g)^\g$ in the base, the fibre in $\Spec \Kir^\lambda(\g)$ is finite and can be identified with the set $\wt (\lambda)$ of weights of $V^\lambda$. By choosing $x\in (\Spec S(\g)^\g)^{\sigma^*}$, we obtain an action of $\sigma^*$ on that fibre, hence an involution on $\wt (\lambda)$. Now if $\sigma$ is (inner to) a split real structure, it is not hard to identify this action on $\wt (\lambda)$ with that of the longest element $w_0$ of the Weyl group. This action has been studied by Stembridge \cite{Stembridge1994_minuscule} and shown to fulfil a $q=-1$ phenomenon. This involves the \emph{Dynkin polynomial} $\D^\lambda$ (see \eqref{eq:dynkinpoly} below), which in the setting of \cite{Stembridge1994_minuscule}, is the rank generating function for the ranked partially ordered set $\wt(\lambda)$. Here, we interpret $\D^\lambda$ as the Poincaré polynomial of $\Kir^\lambda(\g)$ over $S(\g)^\g$, following Panyushev \cite{Panyushev2004_endoj}. Stembridge's $q=-1$ phenomenon then says
\begin{equation}
\label{eq:q=-1}
    \#\wt(\lambda)^{w_0} = \D^\lambda(-1).
\end{equation}

Both sides of this identity have natural interpretations in our setup. Indeed, the left-hand side $\#\wt(\lambda)^{w_0}$ is the number of $\sigma^*$-fixed points in the fibre over a generic $\sigma^*$-fixed base point. Moreover, still assuming that $\sigma$ is inner to a split real structure, we will show that $\sigma^*$ 
acts on $\Kir^\lambda(\g)$ as $(-1)^{\deg}$, that is, it acts like $-1\in \C^\times$ through the $\C^\times$-action corresponding to the natural grading on $\Kir^\lambda(\g)$. The same is then true for the restriction of $\sigma^*$ to the functions on a fibre as above. Moreover, that function ring, denoted $\Kir^\lambda_x(\g)$, inherits the Poincaré polynomial $\D^\lambda$ from $\Kir^\lambda(\g)$ -- now as a graded $\C$-vector space. Thus, $\D^\lambda(-1)$ is simply the trace of $\sigma^*$ on $\Kir^\lambda_x(\g)$. Altogether, we have the following Theorem, which recovers \eqref{eq:q=-1}: 
\begin{thm}
\label{thm:fixed_point_count}
    Let $\g$ be a complex semisimple Lie algebra and $\lambda$ a minuscule weight of $\g$. Further, let $G$ be the connected simply connected Lie group with Lie algebra $\g$, $G^\vee$ its Langlands dual, and $\sigma$ a split real structure of $G^\vee$. Denoting the canonical map $\Spec \Kir^\lambda(\g)\to \Spec S(\g)^\g$ by $\pi$, the fixed-point scheme $(\Spec S(\g)^\g)^{\sigma^*}$ contains a dense open subset $U$ such that the fibre $\pi\inv(x)$ is reduced for each $x\in U$. For such $x$, we then have
    $$
        \D^\lambda(-1)
        =
        \#(\pi\inv(x))^{\sigma^*}
        =
        \#\wt(\lambda)^{w_0},
    $$
    where $w_0$ is the longest element of the Weyl group. Moreover, this quantity is nonzero if and only if $\lambda$ is fixed by a quasi-compact real structure inner to $\sigma$.
\end{thm}

An analogous analysis is possible for involutions induced by non-split inner classes, although the action on $\Kir^\lambda(\g)$ can be more complicated than $(-1)^{\deg}$. Moreover, we expect this work to be closely related to Hausel's \emph{big algebras} \cite{Hausel2024_avatars} (see also Section \ref{sec:background}). These are maximal commutative subalgebras of Kirillov algebras, so they coincide with their ambient Kirillov algebras in the minuscule case treated here, and we expect our results to generalise to all big algebras. In fact, a description of coinvariants similar to Theorem \ref{thm:main} in that generality has recently been obtained for so-called Dynkin automorphisms in \cite{Zveryk2024_dynkin}. Theorem \ref{thm:fixed_point_count} should in this way generalise to arbitrary dominant integral weights, thereby recovering results in \cite{Stembridge1996_canonical}. 

This paper is structured as follows. In Section \ref{sec:background} we review the relevant background on equivariant cohomology and Kirillov algebras and define the coinvariant rings used here. Section \ref{sec:real_str} reviews real structures and describes the induced actions on partial flag varieties and Kirillov algebras. This is then translated to rings of invariant polynomials, which are further discussed in Section \ref{sec:invrings_and_qcrs}. That section also contains a discussion of quasi-compactness. The proof of Theorem \ref{thm:main} is given in Section \ref{sec:main_proof}, and Theorem \ref{thm:degeneracy} is deduced from it in Section \ref{sec:ff}. In Section \ref{sec:weight_action} we describe how our involutions relate to involutions on weights and derive Theorem \ref{thm:fixed_point_count}. Finally, Section \ref{sec:outlook} briefly discusses the aforementioned generalisations to big algebras and arbitrary weights, as well as a connection to the theory of Higgs bundles. Appendix \ref{sec:uniqueness_and_tables} discusses in which sense the real structure appearing in Theorem \ref{thm:main} is unique, and lists the real structures in tables.

\addtocontents{toc}{\SkipTocEntry}
\section*{Acknowledgements}
I would like to thank Tamás Hausel for introducing me to this area of mathematics and for his constant guidance. I would also like to thank Jakub Löwit and Miguel González for fruitful discussions and many helpful comments on this paper. 

This work was done during the author's PhD studies at the Institute of Science and Technology Austria (ISTA). It was funded by the Austrian Science Fund (FWF) 10.55776/P35847.

\addtocontents{toc}{\SkipTocEntry}
\section*{Conventions}
In this paper, all cohomology is taken with complex coefficients. 

Given an action of a group $G$ on a space or algebraic structure $X$, the fixed points are denoted by $X^G$. Similarly, if $f$ is an automorphism of $X$, the fixed points under $f$ are denoted $X^f$. However, if $\g$ is a Lie algebra acting linearly on a vector space $V$, then $V^\g$ denotes the subspace of elements annihilated by $\g$ (so $V^\g =V^G$ for the corresponding action of the connected simply connected Lie group with Lie algebra $\g$). 

Although our main results are for semisimple complex Lie algebras/groups, we will also need to incorporate the reductive case, for which we use the following standard conventions. A Cartan subalgebra $\h$ of a reductive Lie algebra $\g$ is the direct sum of the centre $\z(\g)$ and a Cartan subalgebra $\h^{ss}$ of the semisimple part $\g^{ss}=[\g,\g]$. The roots and Weyl group of $\g$ with respect to $\h$ are defined in terms of $(\g^{ss}, \h^{ss})$ and extended in the obvious way. Thus, if $G$ is reductive with Lie algebra $\g$ and $H\le G$ is the Cartan subgroup corresponding to $\h$, then the Weyl group is isomorphic to $N_G(H)/H$ if $G$ is connected but can be strictly smaller otherwise. 

For semisimple $\g$, we freely use the canonical $\g$-equivariant isomorphism $\g\to \g^*$ afforded by the Killing form. In particular, we freely view the Weyl group with respect to a Cartan subalgebra $\h$ as a subgroup of $\GL(\h)$.

We freely identify (co)weights of a complex Lie group $G$ with their induced (co)weights of the Lie algebra $\g$ (and vice versa when lifts exist). As mentioned previously, we do not specify a choice of Cartan (or Borel) subalgebra when talking about weights, roots, etc., unless necessary. The \emph{weight spaces} of a representation are by convention at least one-dimensional. 

A homomorphism of algebras $A\to B$ is called \emph{free} if it equips $B$ with the structure of a free $A$-module. 

\section{Equivariant cohomology, Kirillov algebras and coinvariant rings}
\label{sec:background}
In this section, we provide further details on relevant background and context. We start with a brief review of equivariant cohomology, a cohomology theory for spaces with group actions.

Let $G$ be a topological group and $X$ a left $G$-space. Let $\mathbb{E}G \to \mathbb{B}G$ be a universal $G$-bundle -- in other words, $\mathbb{E}G$ is a contractible space with free right $G$-action, and $\mathbb{B}G = \mathbb{E}G/G$. ($\mathbb{E}G$ is not unique, but unique up to $G$-homotopy equivalence.) The \emph{(Borel) $G$-equivariant cohomology} of $X$ with coefficients in a ring $R$ is defined as
$$
    H^*_G(X, R)
    \coloneqq
    H^*(X_G, R)
$$
where
$$
    X_G \coloneqq (\mathbb{E}G\times X)/\big( (e\cdot g, x) \sim (e, g\cdot x) \big)
$$
From now on, we always take $R=\C$ and drop this from the notation. It is clear from the definition that $H^*_G(X)$ is a graded $\C$-algebra. In fact, the projection $X\to \pt$ to a point induces a map $X_G\to \pt_G$, making $H^*_G(X)$ canonically an algebra over $H^*_G\coloneqq H^*_G(\pt)$.

Moreover, equivariant cohomology is functorial for morphisms of spaces with group action. That is, let $H$ be another topological group, $Y$ a left $H$-space, $\alpha\colon G\to H$ a group homomorphism, and $f\colon X\to Y$ a map such that
\begin{equation}
\label{eq:eq_coh_functor_condition}
    f(g\cdot x) = \alpha(g)\cdot f(x)
\end{equation}
for all $x\in X$, $g\in G$. Then the pair $(\alpha,f)$ induces a morphism of algebras
\begin{equation}
\label{eq:eq_coh_functor_result}
\begin{tikzcd}
    H^*_H(Y)
        \arrow[r]
    & H^*_G(X)
    \\
    H^*_H
        \arrow[u]
        \arrow[r]
    & H^*_G
        \arrow[u]
\end{tikzcd}
\end{equation}
making $(G,X)\mapsto H^*_G(X)$ into a functor. As a special case, suppose that $Y=X$, $f=\id_X$, and that $\alpha$ is a homotopy equivalence (in addition to being a group homomorphism). One then readily concludes that the resulting homomorphism $H^*_H(X)\to H^*_G(X)$ is an isomorphism. 

Applying this last remark to the case of a point, whenever there is a group homomorphism $G\to H$ which is also a homotopy equivalence we have $H^*_G\cong H^*_H$. In particular, this holds when $H$ is a reductive Lie group and $G$ its maximal compact subgroup (see \eg\ \cite[Prop.\,7.19(a)]{KnappLG}) or when $G$ is a parabolic subgroup of a reductive Lie group and $H$ its Levi factor (\eg\ \cite[Prop.\,7.83(d)]{KnappLG}). If $G$ is a compact Lie group with Lie algebra $\g$, we can describe the ring $H^*_G$ via the Chern-Weil isomorphism \cite[p.\,116]{Chern_complex_mfds}. Namely, $H^*_G$ is concentrated in even degree, meaning $H^{2k+1}_G = 0$ for every $k\in \N$. If we now grade $H^*_G = H^{2*}_G$ by placing $H^{2k}_G$ in degree $k$, then there is a graded isomorphism
\begin{equation}
\label{eq:Chern-Weil}
    H^{2*}_G 
    \cong H^{2*}(\mathbb{B}G)
    \cong S(\g)^G\! \otimes_\R \C.
\end{equation}

Lastly, let us mention the case where $X=G/H$ is a homogeneous space, for $H\le G$ a topological subgroup. In this case, any choice of $\mathbb{E}G$ also has a free $H$-action, so it can be viewed as $\mathbb{E}H$ as well. Moreover, one readily checks that $(G/H)_G\cong \mathbb{E}G/H$, so
\begin{equation}
\label{eq:eq_coh_of_hmg_sp}
    H^*_G(G/H)
    =
    H^*((G/H)_G)
    \cong
    H^*(\mathbb{B}H)
    \cong
    H^*_H.
\end{equation}
For more details on equivariant cohomology we refer to \cite{Carlson_thesis}.

We now collect the key facts on Kirillov algebras stated already in the introduction.
\begin{defn}
    Let $\g$ be a complex semisimple Lie algebra and $\lambda$ a dominant integral weight. Let $V^\lambda$ denote the irreducible $\g$-representation of highest weight $\lambda$, $\End(V^\lambda)$ the $\C$-algebra of linear endomorphisms\footnote{not to be confused with the algebra of $\g$-equivariant endomorphisms (which is isomorphic to $\C$ by Schur's lemma)}, and $S(\g)$ the symmetric algebra of $\g$. The \emph{Kirillov algebra} with label $\lambda$ is the graded subalgebra
    $$
        \Kir^\lambda(\g)
        \coloneqq
        (S(\g)\otimes \End(V^\lambda))^\g
        \subseteq
        S(\g)\otimes \End(V^\lambda)
    $$
    consisting of fixed points of the diagonal $\g$-action.
\end{defn}

If $\h\le \g$ is a Cartan subalgebra and $W$ the Weyl group, Panyushev \cite{Panyushev2004_endoj} defines 
$$
    \C^\lambda(\h)
    \coloneqq
    (S(\h)\otimes \End_\h(V^\lambda))^W,
$$
where $\End_\h(V^\lambda)$ denotes the $\h$-equivariant endomorphisms. Identifying $\g$ with $\g^*$ and $\h$ with $\h^*$ through the Killing form yields an injective restriction map
\begin{equation}
\label{eq:Cartan_restriction_Kirillov_algebra}
    r_\lambda
    \colon
    \Kir^\lambda(\g)
    \hookrightarrow
    \Kir^\lambda(\h).
\end{equation}

\begin{prop}
\label{prop:kirillov_facts}
    Let $\g$ be a complex semisimple Lie algebra and $\lambda$ a dominant integral weight. The Kirillov algebra $\Kir^\lambda(\g)$
    \begin{itemize}
        \item[(i)] has a grading induced from $S(\g)$,
        \item[(ii)] contains $S(\g)^\g$ as the graded subring $S(\g)^\g\otimes \{\id_{V^\lambda}\}$,
        \item[(iii)] is finite-free over $S(\g)^\g$, and
        \item[(iv)] is commutative if and only if $V^\lambda$ is weight-multiplicity free, \ie\ if the weight spaces\footnote{Weight spaces are non-zero by convention.} of $V^\lambda$ are one-dimensional.
        \item[(v)] Moreover, $r_\lambda$ is an isomorphism if and only if $\lambda$ is minuscule. 
    \end{itemize}
\end{prop}
\begin{proof}
    The canonical grading of $S(\g)$ defines a grading of $S(\g)\otimes \End(V^\lambda)$ (in which the second factor contributes trivially). It is preserved by the diagonal $\g$-action and therefore restricts to the grading of $\Kir^\lambda(\g)$ of part (i). Part (ii) is clear. For part (iii), combine \cite[Thm.\,1.1]{Panyushev2004_endoj} with the second paragraph of \cite[p.\,277]{Panyushev2004_endoj}. Part (iv) is \cite[Cor.\,1]{Kirillov2001_introduction}, and part (v) is \cite[Cor.\,2.9]{Panyushev2004_endoj}.
\end{proof}

Recently, Hausel \cite{Hausel2024_avatars}, building on previous work including \cite{Rybnikov2006_argument_shift, FFR2010_irregular_singularity}, has introduced certain maximal commutative subalgebras of Kirillov algebras, called \emph{big algebras}. The minuscule Kirillov algebras studied in this paper are the simplest examples of big algebras, and general big algebras provide important motivation. We now summarise some of their key properties. For every $\lambda$, Hausel's big algebra $\B^\lambda(\g)$ is a maximal commutative graded $S(\g)^\g$-subalgebra of $\Kir^\lambda(\g)$. The morphisms
$$
    \Spec \B^\lambda(\g)
    \to
    \Spec Z(\Kir^\lambda(\g))
    \to
    \Spec S(\g)^\g
$$
(with $Z$ denoting the centre) appear to encode important invariants of $V^\lambda$ such as its Kashiwara crystal \cite[final slide]{Hausel2024_StonyBrook} (see also \cite{HKRW2020_crystals}). Hausel has announced the following geometric model for these algebras:
\begin{thm}[\cf \cite{Hausel2024_avatars}, Theorem 3.1]
\label{thm:geom_model}
    Let $\g$ be a complex semisimple Lie algebra, $G$ the connected simply connected Lie group with Lie algebra $\g$, and $G^\vee$ the Langlands dual group of $G$. For any dominant integral weight $\lambda$ of $\g$, let $\Gr^{\le\lambda}$ denote the affine Schubert variety of $G^\vee$ labelled by $\lambda$. Then we have 
    \begin{itemize}
        \item[(i)] $Z(\Kir^\lambda(\g)) \cong H^{2*}_{G^\vee}(\Gr^{\le\lambda})$
        \item[(ii)] $\B^\lambda(\g) \cong IH^{2*}_{G^\vee}(\Gr^{\le\lambda})$, the $G^\vee$-equivariant intersection cohomology of $\Gr^{\le\lambda}$, as a module over $H^{2*}_{G^\vee}(\Gr^{\le\lambda})$. 
    \end{itemize}
\end{thm}

In this paper, we focus on the minuscule case. Recall that a dominant integral weight $\lambda$ is called \emph{minuscule} if it fulfils one of the following equivalent conditions:
\begin{itemize}
    \item there does not exist a positive root $\alpha$ for which $\lambda - \alpha$ is dominant,
    \item the weights of $V^\lambda$ are all contained in the Weyl group orbit of $\lambda$.
\end{itemize}
Since the highest weight space of $V^\lambda$ is one-dimensional, the second condition shows that $V^\lambda$ is then weight multiplicity free. The minuscule weights form a subset of the fundamental weights; in particular there are finitely many.

\begin{cor}
\label{cor:geom_model}
    Let $\g$, $G^\vee$ and $\lambda$ be as in Theorem \ref{thm:geom_model}. Let $P_\lambda\le G^\vee$ be the parabolic subgroup defined by $\lambda$.\footnote{View $\lambda$ as a cocharacter $\C^\times \to G^\vee$, with derivative $d\lambda\colon \C\to \g$. Then the Lie algebra of $P^\lambda$ is the direct sum of the non-negative eigenspaces of $d\lambda(1)\in \g$, and $P_\lambda$ itself is the normaliser in $G^\vee$ of that Lie algebra.}
    If $\lambda$ is minuscule then $\Kir^\lambda(\g)$ is commutative and isomorphic to $H^{2*}_{G^\vee}(G^\vee / P_\lambda)$.
\end{cor}
\begin{proof}
    Since $\lambda$ is minuscule, the affine Schubert variety $\Gr^{\le\lambda}$ is smooth, and in fact isomorphic (as a $G^\vee$-space) to $G^\vee/P_\lambda$ \cite[Lemma 2.1.13]{Zhu2016_introduction}.
\end{proof}
\begin{rmk}
\label{rmk:geom_model_explicit}
    We appeal to Theorem \ref{thm:geom_model} mostly to contextualise the result. Without appealing to the theorem, an isomorphism $\Kir^\lambda(\g)\cong H^{2*}_{G^\vee}(G^\vee / P_\lambda)$ for minuscule $\lambda$ can be obtained more concretely as follows. For any Cartan subalgebra $\h\le \g$, Panyushev \cite[Thm.\,2.6]{Panyushev2004_endoj} gives an explicit isomorphism $\Kir^\lambda(\g) \cong S(\h)^{W_\lambda}$, where $W_\lambda$ is the stabiliser of $\lambda$ in the Weyl group. The Weyl group is preserved by Langlands duality, and $\h^* \cong \h$ (as a $W$-space) can be viewed as an abstract Cartan subalgebra of the Langlands dual $\g^\vee$. Via the Chevalley restriction theorem (see Theorem \ref{thm:CRT} below) and the Chern-Weil isomorphism \eqref{eq:Chern-Weil}, we then get
    $$
        \Kir^\lambda(\g)
        \cong
        S(\h)^{W_\lambda}
        \cong
        S(\h^*)^{W_\lambda}
        \cong
        S(\el_\lambda)^{\el_\lambda}
        \cong
        H^{2*}_{G^\vee}(G^\vee/P_\lambda),
    $$
    where $\el_\lambda$ is a canonical Levi subalgebra defined by $\lambda$ (see also the discussion at the end of Section \ref{sec:real_str}).
\end{rmk}

We end this section with a brief discussion of fixed points and coinvariant rings. In general, if $X$ is any scheme with a self-morphism $f\colon X\to X$, the \emph{fixed-point scheme} $X^f$ is the largest closed subscheme on which $f$ acts trivially. In other words, $X^f$ is the equaliser of $f$ and $\id_X$ in the category of schemes. In the case of affine schemes, there is a particularly easy description:
\begin{propdef}
\label{prop:coinv_ring}
    Let $\psi$ be an endomorphism of a ring $A$, and $f$ the induced self-morphism of $\Spec A$. Then the fixed-point scheme $(\Spec A)^f$ is given by $\Spec A_\psi\subseteq \Spec A$, where 
    $$
        A_\psi\coloneqq A/(a - \psi(a) \colon a\in A)
    $$
    is the \emph{coinvariant ring}.
\end{propdef}
\begin{proof}
    As a closed subscheme of the affine scheme $\Spec A$, the fixed-point scheme must be of the form $\Spec (A/I)$ for some ideal $I$ of $A$ \cite[\href{https://stacks.math.columbia.edu/tag/01IF}{Tag 01IF}]{stacks-project}. The projection $A\to A/I$ is then a coequaliser of $\psi$ and $\id_A$ in the category of commutative rings, so $I$ is the ideal $(a-\psi(a)\colon a\in A)$ as claimed.
\end{proof}

\section{Action of real structures}
\label{sec:real_str}
Here we review real structures of complex Lie groups and describe their action on partial flag varieties. We then study the resulting involutions on equivariant cohomology. We start by setting up definitions, referring to \cite[\S2--5]{Onishchik2004_real} or \cite[\S2]{GPRa2019_involutions} for details.
\begin{defn}
    A \emph{real form} of a complex Lie algebra $\g$ is a real Lie algebra $\g_0$ whose complexification $\g_0\otimes_\R \C$ is isomorphic to $\g$ (as a complex Lie algebra). 
    A \emph{real structure} of a complex Lie algebra $\g$ is a conjugate-linear Lie algebra automorphism $\sigma\colon \g\to \g$ of order two. 
\end{defn}
If $\sigma$ is a real structure on a complex Lie algebra $\g$, then one readily checks that the fixed points $\g^\sigma$ make up a real form of $\g$. Conversely, for a real form $\g_0$ the complex conjugation on $\g_0\otimes_\R \C$ can be used to define a corresponding real structure on $\g$ (uniquely up to an automorphism of $\g$). This way, real forms and real structures are equivalent. The latter concept generalises more conveniently to complex Lie groups: 
\begin{defn}
    A \emph{real structure} of a complex Lie group $G$ is an antiholomorphic automorphism of order two.
\end{defn}
For our purposes, real structures on the group level are equivalent to those on the Lie algebra level:
\begin{lem}
    Let $G$ be a complex Lie group with Lie algebra $\g$. Any real structure on $G$ defines a real structure on $\g$ by differentiation, and this assignment is injective if $G$ is connected. Conversely, if $G$ is connected and either simply connected or of adjoint type, then any real structure on $\g$ integrates to a unique real structure on $G$. 
\end{lem}
\begin{proof}
    The passage from the group level to the Lie algebra level is standard. It is also clear that real structures of $\g$ lift to $G$ when $G$ is connected and simply connected. For the adjoint type case, it suffices to lift to a simply connected covering group $\Tilde{G}$ and to observe that this lift preserves the centre of $\Tilde{G}$.
\end{proof}

Several equivalence relations on the set of real structures are commonly used (though their names are not entirely standardised):
\begin{defn}[Equivalence relations for real structures]
\label{def:real_str_equiv_classes}
    Let $G$ be a connected complex Lie group and $\g$ its Lie algebra. Let $\sigma_1$, $\sigma_2$ both be real structures on $\g$ or both real structures on $G$.
    \begin{itemize}
        \item[(i)] $\Aut(\g)$ denotes the group of complex Lie algebra automorphisms of $\g$. $\Int(\g)$ denotes its subgroup generated by elements $\exp(\ad(X))$ for $X\in \g$.
        \item[(ii)] $\Aut(G)$ denotes the group of complex Lie group automorphisms of $G$. $\Int(G)$ denotes its subgroup consisting of elements $\Conj_g\coloneqq (h\mapsto ghg\inv)$ for $g\in G$.
        \item[(iii)] We say that $\sigma_1$ and $\sigma_2$ are \emph{isomorphic}, denoted $\sigma_1\approx \sigma_2$, if $\sigma_1 = \psi \sigma_2 \psi\inv$ for $\psi \in \Aut(\g)$ resp.\;$\psi \in \Aut(G)$. This is equivalent to the corresponding real forms being (abstractly) isomorphic.
        \item[(iv)] We say that $\sigma_1$ and $\sigma_2$ are \emph{inner-isomorphic}, denoted $\sigma_1\approx_i \sigma_2$, if $\sigma_1 = \psi \sigma_2 \psi\inv$ for $\psi \in \Int(\g)$ resp.\;$\psi \in \Int(G)$. This is equivalent to the corresponding real forms being isomorphic via conjugation by some $g\in G$.
        \item[(v)] We say that $\sigma_1$ and $\sigma_2$ are \emph{inner (to each other)}, denoted $\sigma_1 \sim_i \sigma_2$, if $\sigma_1 = \psi \sigma_2$ for $\psi \in \Int(\g)$ resp.\;$\psi \in \Int(G)$. Equivalently, $\sigma_1 = \sigma_2 \chi$ for $\chi$ in $\Int(\g)$ resp.\;$\Int(G)$. The corresponding equivalence classes are called \emph{inner classes}.
    \end{itemize}
\end{defn}
Note that inner-isomorphic real structures are automatically isomorphic as well as inner to each other (the latter because $\Int(\g)$ is normal in $\Aut(\g)$). We now come to a classical result of É.\,Cartan \cite{Cartan1914} relating real structures to complex involutions. To state it, we define equivalence relations for such automorphisms analogous to those in Definition \ref{def:real_str_equiv_classes}:
\begin{defn}[Equivalence relations for holomorphic involutions]
    Let $G$ be a connected complex Lie group and $\g$ its Lie algebra. We write $\Aut_2(\g)$ and $\Aut_2(G)$ for the subgroups of involutions in $\Aut(\g)$ resp.\;$\Aut(G)$. Let either $\theta_1,\theta_2 \in\Aut_2(\g)$ or $\theta_1,\theta_2 \in \Aut_2(G)$.
    \begin{itemize}
        \item[(i)] We say that $\theta_1$ and $\theta_2$ are \emph{isomorphic}, denoted $\theta_1\approx \theta_2$, if $\theta_1 = \psi \theta_2 \psi\inv$ for $\psi \in \Aut(\g)$ resp.\;$\psi \in \Aut(G)$.
        \item[(ii)] We say that $\theta_1$ and $\theta_2$ are \emph{inner-isomorphic}, denoted $\theta_1\approx_i \theta_2$, if $\theta_1 = \psi \theta_2 \psi\inv$ for $\psi \in \Int(\g)$ resp.\;$\psi \in \Int(G)$.
        \item[(iii)] We say that $\theta_1$ and $\theta_2$ are \emph{inner (to each other)}, denoted $\theta_1\sim_i \theta_2$, if $\theta_1 = \psi \theta_2$ for $\psi \in \Int(\g)$ resp.\;$\psi \in \Int(G)$. Equivalently, $\theta_1 = \theta_2 \chi$ for $\chi$ in $\Int(\g)$ resp.\;$\Int(G)$. The corresponding equivalence classes are called \emph{inner classes}.
    \end{itemize}
\end{defn}
\begin{thm}[\cf \eg\ \S3 of \cite{Onishchik2004_real}]
\label{thm:Cartan_correspondence}
    Let $\g$ be a complex semisimple Lie algebra. Then $\g$ has a \emph{compact real structure}, \ie\ a real structure $\tau$ such that $\Int(\g^\tau)$ is compact. Moreover, every real structure $\sigma$ of $\g$ is inner-isomorphic to some $\sigma'$ which commutes with $\tau$. Then $\theta\coloneqq \sigma'\tau = \tau\sigma' \in \Aut_2(\g)$. This defines bijections
    \begin{alignat*}{3}
        \{\text{Real structures of $\g$}\}/\!\approx\;&
        \quad
        &&\longleftrightarrow
        \quad
        && \Aut_2(\g)/\!\approx
        \\
        \{\text{Real structures of $\g$}\}/\!\approx_i&
        \quad
        &&\longleftrightarrow
        \quad
        && \Aut_2(\g)/\!\approx_i
        \\
        \{\text{Real structures of $\g$}\}/\!\sim_i&
        \quad
        &&\longleftrightarrow
        \quad
        && \Aut_2(\g)/\!\sim_i
        \\
        [\sigma]&
        \quad
        &&\longleftrightarrow
        \quad
        && [\theta]
    \end{alignat*}
    which do not depend on the choice of $\tau$.
    If $G$ is a connected Lie group with Lie algebra $\g$, then $\tau$ lifts to $G$ and sets up an analogous correspondence between equivalence classes of real structures on $G$ and equivalence classes in $\Aut_2(G)$.
\end{thm}
\begin{defn}
\label{def:Cartan_involution}
    If a real structure $\sigma$ commutes with a compact real structure $\tau$, the corresponding involutive automorphism $\theta = \sigma\tau$ is called the \emph{Cartan involution (of $\sigma$ with respect to $\tau$)}.
\end{defn}

For later use, we also recall the notion of a pinning-preserving automorphism. A \emph{pinning} of a complex reductive Lie algebra $\g$ consists of a Cartan subalgebra $\h\le \g$, a choice of simple roots $\Pi$ for the root system of $(\g,\h)$, and a nonzero root vector $X_\alpha$ for every $\alpha \in \Pi$. 
\begin{defn}
\label{def:pinning}
    Let $\g$ be a complex reductive Lie algebra. An automorphism $\theta\in \Aut(\g)$ is called \emph{pinning-preserving} if there exists a pinning $\{\h, \Pi, \{X_\alpha\}_{\alpha \in \Pi}\}$ of $\g$ such that $\theta$ maps $\h$ to itself and permutes the $X_\alpha$ (hence also the simple roots). If $G$ is a connected Lie group with Lie algebra $\g$, then an automorphism of $G$ is, by definition, \emph{pinning-preserving} if its derivative is.\footnote{This is equivalent to the more common definition involving a pinning of $G$.}
\end{defn}

We now come to the action on partial flag varieties. For the remainder of this section, $G$ will denote a connected complex semisimple Lie group, and $\g$ its Lie algebra. Let $\h\le \g$ be a Cartan subalgebra with corresponding Cartan subgroup $H\coloneqq C_G(\h)\le G$, and let $\Sigma\le \h^*$ denote the root system defined by $(\g,\h)$. Recall that the parabolic subalgebras of $\g$ containing $\h$ can be parameterised by elements of $\h$ as follows: given $v\in \h$, let
$$
    \p_v \coloneqq \h \oplus \!\bigoplus_{\substack{\alpha \in \Sigma\\\alpha(v) \ge 0}} \g_\alpha,
$$
where $\g_\alpha \subset \g$ denotes the root space of $\alpha$. Equivalently, $\p_v$ is the direct sum of non-negative eigenspaces of $\ad(v)$. Each $\p_v$ is a parabolic subalgebra of $\g$ and defines a parabolic subgroup $P_v \coloneqq N_G(\p_v) \le G$. Now if $\lambda\in \Hom(\C^\times, H)$ is a cocharacter, we can take the derivative $d\lambda\colon \C\to \h$ and define
$
    P_\lambda \coloneqq P_{d\lambda(1)}.
$

Now let $\sigma$ be a real structure of $G$ which preserves $H$. This defines an action on the cocharacters given by
$$
    \Hom(\C^\times, H)\to \Hom(\C^\times, H),
    \quad
    \lambda\mapsto \sigma_*\lambda \coloneqq \sigma\circ \lambda \circ \overline{(\cdot)},
$$
with $\overline{(\cdot)}$ denoting complex conjugation on $\C^\times$. One easily verifies that $P_{\sigma_*\lambda} = \sigma(P_\lambda)$. We will be interested in the case where $P_{\sigma_*\lambda}$ is conjugate to $P_\lambda$, which can be expressed using the Weyl group. Recall that the Weyl group $W = N_G(H)/H$ has a natural action on the cocharacters induced by the action of $N_G(H)$. Explicitly, $g\in N_G(H)$ acts as $\lambda \mapsto \Conj_g\circ \lambda$ where $\Conj_g\colon H\to H$ denotes conjugation by $g$. Thus, it makes sense to speak of the \emph{Weyl group orbit} of a cocharacter $\lambda$, and we will focus on the case where $\sigma$ preserves this orbit.

\begin{rmk}(Notions of preserved orbits)
\label{rmk:notions_of_preserving}
\begin{enumerate}
    \item The explicit choice of Cartan subalgebra $\h$ (and Cartan subgroup $H$) is merely for convenience. If no choice is made, a coweight should be viewed as a homomorphism $\lambda\colon \C^\times \to G$ valued in semisimple elements. The action of a real structure $\sigma$ on coweights is defined as above, and now $G$ acts on these by $g\cdot \lambda \coloneqq \Conj_g\circ \lambda$. Given any Cartan subgroup $H$, we can arrange $\sigma(H) = H$ by replacing $\sigma$ with an inner-isomorphic real structure. If $\lambda$ takes values in a Cartan subgroup $H$ preserved by $\sigma$, then $\sigma$ preserves the $G$-orbit of $\lambda$ if and only if it preserves its Weyl group orbit.

    \item Note further that these notions of preserving orbits depend only on the inner class of $\sigma$. More precisely, for $\phi\in \Int(G)$, it is clear that $\sigma\phi$ preserves the $G$-orbit of $\lambda$ if and only if $\sigma$ does. If both $\sigma$ and $\sigma \phi$ preserve a Cartan subgroup $H$ such that $\lambda$ takes values in $H$, then $\phi$ normalises $H$ so that $(\sigma\phi)_*\lambda \in W\cdot \lambda$ if and only if $\sigma_*\lambda \in W\cdot \lambda$. Now, if $B\le G$ is a Borel subgroup containing $H$ such that $\lambda$ is dominant with respect to $B$, there exists a quasi-split real structure $\sigma^{qs}$ inner to $\sigma$ which preserves both $H$ and $B$ (see \eg\ \cite[Problem VI.31]{KnappLG}). Moreover, one can check that there exist a split real structure $\sigma^s$ preserving $H$ and $B$ and a pinning-preserving holomorphic involution $\eta$ such that $\sigma^{qs} = \eta\sigma^s$. We then have $\sigma^s_*\lambda = \lambda$ (irrespective of $\sigma_*\lambda$), so altogether 
    $$
        \sigma_*\lambda \in W\cdot \lambda
        \quad \Leftrightarrow \quad
        \sigma^{qs}_*\lambda = \lambda
        \quad \Leftrightarrow \quad
        \eta\circ \lambda = \lambda.
    $$
\end{enumerate}
\end{rmk}

\begin{prop}
\label{prop:underline_sigma}
    Let $\sigma$ be a real structure of $G$ which preserves the Cartan subgroup $H$ as well as the Weyl group orbit of a cocharacter $\lambda\in \Hom(\C^\times, H)$. Choose $s\in N_G(H)$ such that $\sigma_*\lambda = \Conj_s\circ \lambda$. Then
    $$
        \underline{\sigma}(gP_\lambda)
        \coloneqq
        \sigma(g)sP_\lambda,
        \quad
        g\in G,
    $$
    defines an antiholomorphic involution $\underline{\sigma}$ of $G/P_\lambda$ which does not depend on the choice of $s$.
\end{prop}
\begin{proof}
    Firstly, for $g\in G$ and $h\in P_\lambda$ we have
    $$
        \sigma(gh)sP_\lambda
        =
        \sigma(g)\sigma(h)sP_\lambda
        =
        \sigma(g)s \Conj_{s\inv}(\sigma(h)) P_\lambda
        =
        \sigma(g)s P_\lambda
    $$
    because $\Conj_s(P_\lambda) = \sigma(P_\lambda)$, so $\underline{\sigma}(gP_\lambda)$ is well-defined. 
    
    Secondly, for another $s'\in N_G(H)$ with $\Conj_{s'}\circ \lambda = \sigma_*\lambda$, we have
    $$
        s'P_\lambda
        =
        s(s\inv s')P_\lambda
        =
        sP_\lambda
    $$
    since $\Conj_{s\inv s'}(P_\lambda) = P_\lambda$ and $P_\lambda$ is self-normalising (forcing $s\inv s'\in P_\lambda)$. This shows that $\underline{\sigma}$ does not depend on the choice of $s$.

    Lastly, the assumption $\Conj_s \circ \lambda = \sigma_*\lambda$ implies $\Conj_{\sigma(s)s}\circ \lambda = \lambda$. As before, we conclude that $\sigma(s)s\in P_\lambda$. Then
    $$
        \underline{\sigma}^2(gP_\lambda)
        =
        g\sigma(s)sP_\lambda
        =
        gP_\lambda,
        \quad
        g\in G
    $$
    verifies that $\underline{\sigma}$ is indeed an involution. Antiholomorphicity is easily checked by lifting $\underline{\sigma}$ to $G$.
\end{proof}
\begin{rmk}
\label{rmk:real_str_on_affGr}
    The real structure $\sigma$ induces a real structure $\widetilde{\sigma}$ on the affine Grassmannian $\Gr = \Gr_G$ of $G$. The affine Schubert varieties of $\Gr$ can be labelled by $W$-orbits of coweights, and $\widetilde{\sigma}$ permutes the affine Schubert varieties as $\sigma_*$ permutes these $W$-orbits. In particular, $\widetilde{\sigma}$ preserves the affine Schubert variety $\Gr^{\le \lambda}$ if and only if $\sigma_*\lambda \in W\cdot \lambda$. If this is the case, $\widetilde{\sigma}$ restricts to a real structure of $\Gr^{\le \lambda}$. Now, if $\lambda$ is in addition minuscule, then $\Gr^{\le \lambda}$ is canonically isomorphic to $G/P_\lambda$, and the isomorphism identifies $\widetilde{\sigma}$ with the real structure $\underline{\sigma}$ of Proposition \ref{prop:underline_sigma}.
\end{rmk}

As an antiholomorphic involution, $\underline{\sigma}$ fixes a real submanifold of $G/P_\lambda$ whose real dimension is $\dim_\C G/P_\lambda$. However, it turns out that this submanifold can be empty:
\begin{prop}
Let $G$, $\sigma$, and $\lambda$ be as in Proposition \ref{prop:underline_sigma}. Then
    $$
        (G/P_\lambda)^{\underline{\sigma}}
        =
        \begin{cases}
            G^\sigma/P_\lambda^\sigma
                & \text{if } \sigma_*\lambda = \lambda \\
            \emptyset
                & \text{otherwise.}
        \end{cases}
    $$
\end{prop}
\begin{proof}
    The key observation is that $(G/P_\lambda)^{\underline{\sigma}}$ is a closed union of $G^\sigma$-orbits in $G/P_\lambda$. Thus, it is either empty or contains the distinguished orbit through $eP_\lambda$ (where $e\in G$ is the neutral element) \cite[Cor.\,3.4]{Wolf1969_real_action}. But $\underline{\sigma}(eP_\lambda) = sP_\lambda$ for $s\in N_G(H)$ with $\Conj_s\circ \lambda = \sigma_*\lambda$. This equals $eP_\lambda$ precisely when $s\in P_\lambda$, which is equivalent to $\sigma_*\lambda = \lambda$. 
    
    In case $\sigma_*\lambda = \lambda$, we have seen that the fixed submanifold contains the distinguished orbit $G^\sigma\cdot (eP_\lambda) = G^\sigma/P_\lambda^\sigma$. But $\dim_\R(G/P_\lambda)^{\underline{\sigma}} = \dim_\C (G/P_\lambda)$ is the minimal dimension of $G^\sigma$-orbits \cite[Thm.\,3.6]{Wolf1969_real_action}, which is attained only by the distinguished orbit \cite[Cor.\,3.4]{Wolf1969_real_action}, so the fixed point set consists of that orbit alone.
\end{proof}

We will be interested in the action of $\underline{\sigma}$ on equivariant cohomology. More precisely, the pair $(\sigma, \underline{\sigma})$ induces an algebra involution
\begin{equation}
\label{eq:sigma_on_eq_coh}
\begin{tikzcd}
    H^{2*}_G(G/P_\lambda)
        \arrow[r, "\sigma^*"]
    & H^{2*}_G(G/P_\lambda)
    \\
    H^{2*}_G
        \arrow[u]
        \arrow[r, "\sigma^*"]
    & H^{2*}_G
        \arrow[u]
\end{tikzcd} 
\end{equation}
as in \eqref{eq:eq_coh_functor_condition}--\eqref{eq:eq_coh_functor_result}. We finish this section with a description of $\sigma^*$.
\begin{lem}
\label{lem:inner_class_only}
    Let $G$, $\sigma$, and $\lambda$ be as in Proposition \ref{prop:underline_sigma}. The involution $\sigma^*$ of \eqref{eq:sigma_on_eq_coh} depends only on the inner class of $\sigma$.\footnote{This lemma is a special case of Remark A.7(i) in \cite{Zhu2012_geometric}.}
\end{lem}
\begin{proof}
    If $\psi = \Conj_u\in \Int(G)$ is such that $\sigma'=\psi\sigma$ is another real structure, one finds that $\underline{\sigma}$ and $\underline{\sigma'}$ differ by the action $\ell_u$ of $u$ on $G/P_\lambda$ by left multiplication. Since $G$ is connected, this action is homotopic to the identity. It follows that $(\sigma')^* = \sigma^*$.
\end{proof}

Using \eqref{eq:eq_coh_of_hmg_sp}, we find that $H^{2*}_G(G/P_\lambda) \cong H^{2*}_{P_\lambda}$. Moreover, $P_\lambda$ has a canonical Levi subgroup $L_\lambda$ whose Lie algebra $\el_\lambda$ is the centraliser of $d\lambda(1)$ in $\g$. By the discussion after \eqref{eq:eq_coh_functor_result}, we then further have
$
    H^{2*}_G(G/P_\lambda) \cong H^{2*}_{L_\lambda}.
$
Now let $\tau$ be a compact real structure of $G$ which preserves $H$ and commutes with $\sigma$ (which exists \eg\ by \cite[Prop.\,II.6]{Onishchik2004_real}). Then $L_\lambda^\tau$ is a compact real form and hence a maximal compact subgroup of $L_\lambda$, so $H^{2*}_{L_\lambda} \cong H^{2*}_{L_\lambda^\tau}$. At this point, we can apply the isomorphism \eqref{eq:Chern-Weil} to conclude
\begin{equation}
\label{eq:long_iso_Levi_level}
    H^{2*}_G(G/P_\lambda) 
    \cong H^{2*}_{P_\lambda}
    \cong H^{2*}_{L_\lambda}
    \cong H^{2*}_{L_\lambda^\tau}
    \cong S(\el_\lambda^\tau)^{L_\lambda^\tau} \otimes_\R \C
    \cong S(\el_\lambda)^{L_\lambda}
\end{equation}
(where $\tau$ also denotes the real structure on $\g$ obtained from $\tau$ by differentiation). Similarly, we have
\begin{equation}
\label{eq:long_iso_base_level}
    H^{2*}_G
    \cong H^{2*}_{G^\tau}
    \cong S(\g^\tau)^{G^\tau} \otimes_\R \C
    \cong S(\g)^G.
\end{equation}
Lastly, via the Killing form on $\g$ we obtain equivariant isomorphisms $\g\cong \g^*$ and $\el_\lambda \cong \el_\lambda^*$. We interpret the ring $S(\g^*)^G$ as the \emph{$G$-invariant polynomials} on $\g$ and thus denote it as $\C[\g]^G$ (and extend this notation to other Lie algebras and Lie groups). Then:
\begin{equation}
\label{eq:killing_ring_isos}
    S(\g)^G \cong \C[\g]^G,
    \quad
    S(\el_\lambda)^{L_\lambda} \cong \C[\el_\lambda]^{L_\lambda}
\end{equation}
\begin{lem}
\label{lem:from_eq_coh_to_invrings}
    Let $G$, $\sigma$, and $\lambda$ be as in Proposition \ref{prop:underline_sigma} and assume that $\sigma_*\lambda = \lambda$. Under the isomorphisms \eqref{eq:long_iso_Levi_level}, \eqref{eq:long_iso_base_level} and \eqref{eq:killing_ring_isos}, the involution $\sigma^*$ of \eqref{eq:sigma_on_eq_coh} is identified with
    $$
    \begin{tikzcd}
        \C[\el_\lambda]^{L_\lambda}
            \arrow[r, "\theta^*"]
        & \C[\el_\lambda]^{L_\lambda} \\
        \C[\g]^G
            \arrow[u, "\res"]
            \arrow[r, "\theta^*"]
        & \C[\g]^G,
        \arrow[u, "\res"]
    \end{tikzcd}
    $$
    where $\theta$ is the Cartan involution as in Definition \ref{def:Cartan_involution}, $\theta^*$ denotes precomposition of polynomials with $\theta$ and $\res$ their restriction to $\el_\lambda$.  
\end{lem}
\begin{proof}
    Under our assumption, the map $\underline{\sigma}$ has the simple form
    $$
        \underline{\sigma}(gP_\lambda) = \sigma(g)P_\lambda,
        \quad
        g\in G.
    $$
    It is then not difficult to trace this through each of the steps in \eqref{eq:long_iso_Levi_level}--\eqref{eq:killing_ring_isos} to arrive at the claimed description.
\end{proof}

The description in Lemma \ref{lem:from_eq_coh_to_invrings} will be used for the proof of Theorem \ref{thm:main}. However, translating this further through the isomorphism $\C[\el_\lambda]^{L_\lambda} \cong \Kir^\lambda(\g) $, the involutions $\sigma^*$ turn out to fit into a more general class of finite order automorphisms of $\Kir^\lambda(\g)$, which we now introduce.\footnote{We note that their construction, in Proposition \ref{prop:iota_eta_zeta} below, closely resembles that in \cite[Thm.\,6.3]{GPRa2019_involutions}.} The main idea is that two natural sources of automorphisms are (suitable) automorphisms of $\g$ and ``rescaling the argument'' on $S(\g) \cong \C[\g]$. 
\begin{prop}
\label{prop:iota_eta_zeta}
    Let $\g$ be a complex semisimple Lie algebra, $\h$ a Cartan subalgebra, and $\lambda$ be a dominant integral weight (with respect to some choice of positivity) with corresponding irreducible representation $\rho^\lambda \colon \g\to \End(V^\lambda)$. Suppose given
    \begin{itemize}
        \item an automorphism $\eta\in \Aut(\g)$ such that $\eta(\h)=\h$, and
        \item a root of unity $\zeta\in \C^\times$.\footnote{The importance of this parameter was explained to us by Miguel González.}
    \end{itemize}
    \begin{itemize}
        \item[(i)] The map $\rho^\lambda\circ \eta$ defines another irreducible representation on $V^\lambda$. Its highest weight, which we denote $\eta^*\lambda$, is conjugate to $\lambda\circ \eta$ through the Weyl group.
        \item[(ii)] Up to rescaling, there is a unique linear isomorphism $A = A_{\eta,\lambda} \colon V^\lambda \to V^{\eta^*\lambda}$ such that
        $$
            A(x\cdot v) = \eta(x) \cdot A(v)
            \quad
            \forall x\in \g, v\in V^\lambda.
        $$
        \item[(iii)] If $\eta^*\lambda = \lambda$, then 
        $$
            \widetilde{\iota_{\eta, \zeta}}
            \coloneqq S (\zeta\eta) \otimes \Conj_{A}
            \colon\,\,
            S(\g) \otimes \End(V^\lambda) \to S(\g) \otimes \End(V^\lambda)
        $$
        restricts to an automorphism $\iota_{\eta,\zeta}$ of $\Kir^\lambda(\g)$. (Here $S(\zeta\eta)$ denotes the automorphism of $S(\g)$ induced by the linear isomorphism $x\mapsto \zeta\cdot\eta(x)$, and $\Conj_A$ is conjugation by $A=A_{\eta,\lambda}$.)
        \item[(iv)] Let $\psi \in \Int(\g)$ normalise $\h$ and put $\eta'\coloneqq \eta\psi$. Then $\eta^*\lambda = (\eta')^*\lambda$, and if this equals $\lambda$ we have $\iota_{\eta',\zeta} = \iota_{\eta, \zeta}$.
        \item[(v)] If $\eta\in \Int(\g)$, then $\eta^*\lambda = \lambda$ for all dominant integral weights $\lambda$, and $\iota_{\eta,\zeta} = \iota_{\id,\zeta} = \zeta^{\deg}$. That is, $\iota_{\eta,\zeta}$ acts on homogeneous elements of degree $k$ by multiplication with $\zeta^k$.
    \end{itemize}
\end{prop}
\begin{proof}
    The first item is straightforward, and the second an easy consequence of Schur's Lemma. For (iii), it suffices to observe that $\widetilde{\iota_{\eta,\zeta}}$ intertwines the usual $\g$-action on both factors with the action twisted by $\eta$; in particular it sends invariant elements to invariant elements. For item (iv) we use that $\lambda \circ \psi$ is in the same Weyl group orbit as $\lambda$, and then the definition of $\Kir^\lambda(\g)$ to see that the contribution of $\psi$ cancels out. Lastly, (v) is an application of (iv) followed by a straightforward computation of $\iota_{\id, \zeta}$. 
\end{proof}

\begin{prop}
\label{prop:sigma_action_translated}
    Let $\g$ be a complex semisimple Lie algebra, $G$ the connected simply connected Lie group with Lie algebra $\g$, $G^\vee$ the Langlands dual group of $G$ and $\g^\vee$ its Lie algebra. Let $\lambda$ be a minuscule weight of $\g$, which we view also as a cocharacter of $G^\vee$, and $\sigma$ a real structure of $G^\vee$ preserving the $G^\vee$-orbit of $\lambda$. Then there exists $\eta\in \Aut_2(\g)$ such that the involution $\sigma^*$ of $H^{2*}_{G^\vee}(G^\vee/P_\lambda)$ defined as in \eqref{eq:sigma_on_eq_coh} fits into a commutative diagram
    $$
    \begin{tikzcd}
        H^{2*}_{G^\vee}(G^\vee/P_\lambda)
            \arrow[r, "\sigma^*"]
        & H^{2*}_{G^\vee}(G^\vee/P_\lambda)
        \\
        \Kir^\lambda(\g)
            \arrow[r, "\iota_{\eta, -1}"]
            \arrow[u, "\cong"]
        &
        \Kir^\lambda(\g).
            \arrow[u, "\cong"]
    \end{tikzcd}
    $$

    An explicit construction of $\eta$ is as follows: let $\sigma^{s}$ be any split real structure of $G^\vee$, and let $\kappa$ denote the image of $\sigma\sigma^s$ in $\Out(\g^\vee)\coloneqq \Aut(\g^\vee)/\Int(\g^\vee)$. Use the canonical isomorphism $\Out(\g)\cong \Out(\g^\vee)$, to identify $\kappa$ with an element of $\Out(\g)$, and define $\eta\in \Aut(\g)$ to be any lift (of order two) of that element. 
\end{prop}
\begin{proof}
    Let $\b\le \g$ be a Borel subalgebra and $\h\le \b$ a Cartan subalgebra; we may realise $\lambda$ as an element of $\h^*$ dominant with respect to $\b$ and identify $\h^*$ with a Cartan subalgebra of $\g^\vee$ preserved by $\sigma$.
    To keep track of identifications, we divide the isomorphism $H^{2*}_{G^\vee}(G^\vee/P_\lambda) \cong \Kir^\lambda(\g)$ sketched in Remark \ref{rmk:geom_model_explicit} into numbered steps (recalled below):
    $$
        H^{2*}_{G^\vee}(G^\vee/P_\lambda)
        \overset{\mathbf{1}}{\cong}
        \C[\el_\lambda]^{L_\lambda}
        \overset{\mathbf{2}}{\cong}
        \C[\h^*]^{W_\lambda}
        \overset{\mathbf{3}}{\cong}
        (S(\h) \otimes \End_\h(V^\lambda))^W
        \overset{\mathbf{4}}{\cong}
        \Kir^\lambda(\g).
    $$
    
    By Remark \ref{rmk:notions_of_preserving}(2), there exists a quasi-split real structure $\sigma^{qs}$ inner to $\sigma$ such that $\sigma^{qs}(\h^*) = \h^*$ and $\sigma^{qs}_*\lambda = \lambda$. Let $\theta$ be the Cartan involution of $\sigma^{qs}$. By Lemma \ref{lem:from_eq_coh_to_invrings}, step $\mathbf{1}$ (\cf \eqref{eq:long_iso_Levi_level}--\eqref{eq:killing_ring_isos}) then identifies $\sigma^*$ with $\theta^*$. Step $\mathbf{2}$ is the Chevalley restriction theorem for the connected reductive group $L_\lambda$ and the Cartan subalgebra $\h^*\le \el_\lambda$ (see Corollary \ref{cor:crt_extended} below). In particular, this step identifies $\theta^*$ with $(\theta|_{\h^*})^*$, which we still denote $\theta^*$ by abuse of notation. 

    Now, if $\sigma^{qs}$ is split then $\theta$ acts on $\h^*$ by $-1$. Since all isomorphisms above preserve the gradings, they identify $\sigma^*$ with $(-1)^{\deg} = \iota_{\id, -1}$ in this case. More generally, there exists a pinning-preserving involution $\widetilde{\eta}$ of $\g^\vee$ such that $\widetilde{\eta}(\h^*) = \h^*$ and $\sigma^s \coloneqq \widetilde{\eta}\sigma^{qs}$ is a split real structure.\footnote{To obtain $\eta^\vee$, use the action of $\sigma^{qs}$ on simple roots with respect to a Borel subalgebra fixed by $\sigma^{qs}$.} It therefore remains to translate the involution $\widetilde{\eta}^*$ through steps $\mathbf{3}$ and $\mathbf{4}$.

    Step $\mathbf{3}$ consists of noticing that the ring $\End_\h(V^\lambda)$ (the $\h$-equivariant endomorphisms of $V^\lambda$) is isomorphic to $\bigoplus_\mu\End(V^\lambda_\mu)$ where $V^\lambda_\mu$ are the weight spaces. Hence,
    $$
        \Kir^\lambda(\h) \overset{\text{def}}{=}
        (S(\h) \otimes \End_\h(V^\lambda))^W
        \cong
        (S(\h) \otimes \bigoplus \End(V^\lambda_\mu))^W,
    $$
    and since $W$ acts transitively on the one-dimensional summands $\End(V^\lambda_\mu)$, any $W$-invariant element of $S(\h) \otimes \bigoplus \End(V^\lambda_\mu)$ is determined by its component in $S(\h) \otimes \End(V^\lambda_\lambda)$. Restriction to this component yields an isomorphism $\Kir^\lambda(\h) \to S(\h)^{W_\lambda} = \C[\h^*]^{W_\lambda}$, the inverse of isomorphism $\mathbf{3}$ (cf.\,\cite[Thm.\,2.6]{Panyushev2004_endoj}). The isomorphism $\mathbf{4}$ is simply the inverse of the restriction map $r_\lambda$ from \eqref{eq:Cartan_restriction_Kirillov_algebra} (cf.\,Proposition \ref{prop:kirillov_facts}(v)). 
    
    We claim that steps $\mathbf{3}$ and $\mathbf{4}$ identify $\widetilde{\eta}^*$ with $\iota_{\eta,1}$, where $\eta\in \Aut(\g)$ preserves $\h$ and lifts the class of $\widetilde{\eta}$ in $\Out(\g^\vee)\cong \Out(\g)$. To see this, let $f\in \Kir^\lambda(\g)$ be arbitrary. We may assume $\eta$ to be pinning-preserving, so that $\lambda\circ \eta = \lambda$ and $\Conj_A$ acts trivially on $\End(V^\lambda_\lambda)$. If the $\End(V^\lambda_\lambda)$ component of $r_\lambda(f)$ is $f_\lambda$, then that of $r_\lambda(\iota_{\eta,1}(f))$ is $S(\eta)(f_\lambda)$. Viewing $S(\h)$ as $\C[\h^*]$, this is precisely $\widetilde{\eta}^*f_\lambda$, which proves our claim. 
    
    Altogether, the isomorphisms $\mathbf{1}$ through $\mathbf{4}$ identify $\sigma^* = (\sigma^s\widetilde{\eta})^*$ with
    $
        \iota_{\id, -1} \circ\iota_{\eta,1}
        =
        \iota_{\eta, -1}
    $
    as we wanted to show.
\end{proof}

\section{Invariant rings and quasi-compact real structures}
\label{sec:invrings_and_qcrs}
Lemma \ref{lem:from_eq_coh_to_invrings} translates the involutions \eqref{eq:sigma_on_eq_coh} we are interested in to the setting of invariant polynomials. In this section, we collect important facts about invariant polynomials, most importantly the key Lemma \ref{lem:dist_coinv}. This result motivates the use of quasi-compact real structures, as discussed at the end of the section. We begin by reiterating a definition from the previous section:
\begin{defn}
    Let $G$ be a complex Lie group with Lie algebra $\g$. Then $\C[\g]^G$ denotes the subring of $\C[\g] \coloneqq S(\g^*)$ consisting of elements invariant under the canonical $G$-action. Analogously, $\C[\g]^\g$ denotes the subring of $\C[\g]$ consisting of elements annihilated by $\g$.
\end{defn}
\begin{rmk}
    By definition of $\Int(\g)$, we have $\C[\g]^\g = \C[\g]^{\Int(\g)}$. If $G$ is connected, the adjoint action $G\to \Aut(\g)$ has values in $\Int(\g)$, so $\C[\g]^G = \C[\g]^\g$. For disconnected $G$, $\C[\g]^G$ can be strictly smaller than $\C[\g]^\g$, see Example \ref{expl:typeAexpl} below.
\end{rmk}

An important classical result about invariant polynomials is the following theorem of Chevalley:
\begin{thm}[\cf\cite{Varadarajan_Lie}, Thm.\,4.9.2]
\label{thm:CRT}
    Let $\g$ be a semisimple complex Lie algebra, $\h\le \g$ a Cartan subalgebra, and $W$ the corresponding Weyl group. Then the canonical restriction map $\C[\g]^\g\to \C[\h]^W \coloneqq S(\h^*)^W$ is an isomorphism.
\end{thm}
\begin{cor}
\label{cor:crt_extended}
    Let $G$ be a complex reductive Lie group with Lie algebra $\g$. Let $\h\le \g$ be a Cartan subalgebra, and $N_G(\h)$ its normaliser. Then restriction defines an isomorphism
    $$
        \C[\g]^G \cong \C[\h]^{N_G(\h)}.
    $$
\end{cor}
\begin{proof}[Proof of Corollary]
    Let $G_0$ denote the identity component of $G$. Using the usual decomposition of $\g$ into its centre and derived subalgebra, Theorem \ref{thm:CRT} extends at once to
    $$
        \C[\g]^{G_0} \cong \C[\h]^W.
    $$
    To obtain the $G-$invariants, we now have to take into account the action of the component group $G/G_0$. But it is easy to see that $N_G(\h)$ meets all components of $G$, so
    $$
        \C[\g]^G
        =
        (\C[\g]^{G_0})^{G/G_0}
        =
        (\C[\g]^{G_0})^{N_G(\h)}
        \cong
        (\C[\h]^W)^{N_G(\h)}
        =
        \C[\h]^{N_G(\h)}.
    $$
\end{proof}

Having recalled this tool, we now return to the study of involutions on invariant polynomial rings. Directly from the definition, we obtain the following counterpart of Lemma \ref{lem:inner_class_only}:
\begin{prop}
    Let $\g$ be a complex Lie algebra and $\theta\in \Aut_2(\g)$. Then the involution $\theta^*$ of $\C[\g]^\g$ depends only on the inner class of $\theta$.
\end{prop}
This suggests to look for particularly well-behaved involutions in a given inner class, and a natural candidate are those that preserve a pinning (see Definition \ref{def:pinning}). We recall some well-known results in this context:
\begin{lem}
\label{lem:fixed_subalgebras}
    Let $\g$ be a complex reductive Lie algebra and $\theta \in \Aut(\g)$.
    \begin{itemize}
        \item[(i)] The fixed subalgebra $\g^\theta$ is reductive. 
        
        \item[(ii)] If $\h$ is a Cartan subalgebra of $\g$ preserved by $\theta$ such that $\dim \h^\theta$ is maximal among such Cartan subalgebras, then $\h^\theta$ is a Cartan subalgebra of $\g^\theta$. 
        
        \item[(iii)] If $\g$ is semisimple and $\theta$ preserves a pinning $(\h, \Pi, \{X_\alpha\}_{\alpha\in \Pi})$, then $\g^\theta$ is semisimple with Cartan subalgebra $\h^\theta$. Moreover, if $W$ is the Weyl group of $(\g,\h)$, then its subgroup
        $$
            W^\theta \coloneqq \{w\in W\colon w\theta = \theta w\}
        $$
        is identified with the Weyl group of $(\g^\theta, \h^\theta)$ by restriction to $\h^\theta$.
    \end{itemize}
\end{lem}
\begin{proof}
    For part (i), note that $\theta$ preserves the direct sum decomposition of $\g$ into its centre $\z(\g)$ and derived subalgebra $[\g,\g]$. The Killing form of $[\g,\g]$ restricts to a nondegenerate $\ad$-invariant symmetric bilinear form of $[\g,\g]^\theta$, which shows that $[\g,\g]^\theta$ is reductive. It follows that $\g^\theta = \z(\g)^\theta \oplus [\g,\g]^\theta$ is reductive as well. 

    To see that $\h^\theta$ of part (ii) is a Cartan subalgebra, one can for instance use Gantmacher's normal form, \cf\ \cite[Thm.\,4.2]{Onishchik2004_real}, relating $\theta$ to a pinning-preserving automorphism.
    
    For part (iii), we refer to \cite[ch.\,11]{Steinberg1967_Chevalley_groups}.
\end{proof}
\begin{lem}
\label{lem:dist_coinv}
    Let $\g$ be a complex reductive Lie algebra and $\theta$ a pinning-preserving automorphism of $\g$. Then the restriction map
    $
        \C[\g]^\g \to \C[\g^\theta]^{\g^\theta}
    $
    is surjective and induces an isomorphism of $\C[\g^\theta]^{\g^\theta}$ with the coinvariant ring $\C[\g]^\g_{\theta^*}$. 
\end{lem}
\begin{proof}
Let $\g_{\text{der}}\coloneqq [\g,\g]$ be the derived subalgebra. The decomposition $\g = \z(\g) \oplus \g_{\text{der}}$ is preserved by $\theta$ and yields
$$
    \C[\g]^\g \cong \C[\z(\g)]\otimes \C[\g_{\text{der}}]^{\g_{\text{der}}}.
$$
Clearly, it then suffices to prove the lemma for $\z(\g)$ and $\g_{\text{der}}$ separately. For the affine space $\z(\g)$ it follows immediately from Proposition \ref{prop:coinv_ring}, so for the remainder we may assume that $\g$ is semisimple.

Now let $(\h, \Pi, \{X_\alpha\}_{\alpha\in \Pi})$ be a pinning of $\g$ preserved by $\theta$. By Lemma \ref{lem:fixed_subalgebras}(iii), $\h^\theta$ is then a Cartan subalgebra of $\g^\theta$. Moreover, the sum $e\coloneqq \sum_{\alpha} X_\alpha$ is fixed by $\theta$ and is a principal nilpotent element (\cf \cite[\S5]{Kostant1959_TDS}) for both $\g$ and $\g^\theta$. Upon extending it to an $\sl_2$-triple $(e,f,h)$ in $\g^\theta$, we obtain Kostant sections $\s\coloneqq e+ \g_f$ of $\g$ and $e+\g^\theta_f = \s^\theta$ for $\g^\theta$. By \cite[Thm.\,7]{Kostant1963_LGRoPR}, the restriction map $\C[\g]^\g\to \C[\g^\theta]^{\g^\theta}$ is then equivalent to the restriction $\C[\s]\to \C[\s^\theta]$. This is clearly surjective, and Proposition \ref{prop:coinv_ring} identifies $\C[\s^\theta]$ with the coinvariant ring $\C[\s]_\theta \cong \C[\g]^\g_{\theta^*}$.
\end{proof}
\begin{cor}
\label{cor:disconn_inv_ring_nice}
    Let $G$ be a connected complex reductive Lie group with Lie algebra $\g$. If $\theta$ is a pinning-preserving automorphism of $G$ (with derivative also denoted by $\theta$), then
    $$
        \C[\g^\theta]^{G^\theta} = \C[\g^\theta]^{\g^\theta}.
    $$
\end{cor}
\begin{proof}
    We have to show that every $g\in G^\theta$ acts trivially on every $f\in \C[\g^\theta]^{\g^\theta}$. By Lemma \ref{lem:dist_coinv}, $f$ admits an extension $\Tilde{f}\in \C[\g]^\g$ with $\Tilde{f}|_{\g^\theta} = f$. Moreover, since $G$ is connected, we have $\C[\g]^\g = \C[\g]^G$, so $g\cdot \Tilde{f} = \Tilde{f}$. But this implies $g\cdot f  = f$, too.
\end{proof}
The point of the preceding corollary is that it holds even though $G^\theta$ need not be connected. If $\theta$ is not pinning-preserving, the following example shows that $\C[\g^\theta]^{G^\theta}$ can indeed be strictly smaller than $\C[\g^\theta]^{\g^\theta}$.
\begin{expl}
\label{expl:typeAexpl}
    Let $G = \GL_{2n}(\C)$ for some $n\in \N$, and let $\theta$ be inverse-transpose, $\theta(A) = (A^t)\inv$. Then $G^\theta = \O_{2n}(\C)$ with Lie algebra $\g^\theta = \so_{2n}(\C)$. Using the Chevalley restriction theorem, one finds that
    $$
        \C[\so_{2n}(\C)]^{\so_{2n}(\C)}
        \cong
        \C[x_1,\ldots, x_n]^{S_n\ltimes \Z_2^{n-1}}
    $$
    where $S_n\ltimes \Z_2^{n-1}$ acts on the variables by signed permutations with an even number of sign changes. From this ring, we obtain $\C[\so_{2n}(\C)]^{\O_{2n}(\C)}$ by taking into account the component group of $\O_{2n}(\C)$, which has order 2. Its nontrivial element acts on $\C[x_1,\ldots, x_n]^{S_n\ltimes \Z_2^{n-1}}$ by a single sign change, so that
    $$
        \C[\so_{2n}(\C)]^{\O_{2n}}
        \cong 
        \C[x_1,\ldots, x_n]^{S_n\ltimes \Z_2^n},
    $$
    with $S_n\ltimes \Z_2^n$ acting by arbitrary signed permutations. This is a strict subring of $\C[\so_{2n}(\C)]^{\so_{2n}(\C)}$ -- for instance, $x_1x_2\cdots x_n\not \in \C[x_1,\ldots, x_n]^{S_n\ltimes \Z_2^n}$.
\end{expl}
We can now translate back to the setting of real structures:
\begin{cor}
\label{cor:qc_main}
    Let $\sigma$ be a real structure of a connected complex reductive group $G$ with Cartan involution $\theta$. If $\theta$ is pinning-preserving then the restriction
    $
        H^{2*}_G \to H^{2*}_{G^\sigma}
    $
    is surjective and identifies $H^{2*}_{G^\sigma}$ with the coinvariant ring $(H^{2*}_G)_{\sigma^*}$
\end{cor}
\begin{proof}
    As in \eqref{eq:long_iso_base_level}--\eqref{eq:killing_ring_isos} we obtain compatible isomorphisms
    $$
        H^{2*}_G \cong \C[\g]^G,
        \quad
        H^{2*}_{G^\sigma} \cong \C[\g^\theta]^{G^\theta},
    $$
    where $\theta$ is the Cartan involution of $\sigma$ (with respect to a suitable compact real structure). The result then follows from Lemma \ref{lem:dist_coinv} in combination with Corollary \ref{cor:disconn_inv_ring_nice}.
\end{proof}
\begin{defn}
\label{def:qcrs}
    A real structure with pinning-preserving Cartan involution (as in Corollary \ref{cor:qc_main}) is called \emph{quasi-compact}.
\end{defn}

Quasi-compact real structures play a key role in this paper due to Corollary \ref{cor:qc_main}. The next Proposition establishes basic facts about them, including the reason for their name.\footnote{While uncommon, the name \emph{quasi-compact} is not new, appearing for example in \cite[Rmk.\,8.3]{AT_2018_Galois_and_Cartan}.}
\begin{prop}
\label{prop:qc_funfacts}
    Let $G$ be a connected complex reductive group. Every inner class of real structures on $G$ contains a quasi-compact real structure and this real structure is unique up to inner-isomorphism. If $G_0$ is the corresponding real form, then the dimension of its maximal compact subgroup is maximal among all real forms in the given inner class.
\end{prop}
\begin{proof}
    Existence and uniqueness of the quasi-compact real structure are equivalent via Theorem \ref{thm:Cartan_correspondence} to corresponding statements about pinning-preserving involutions. In turn, these follow from well-known descriptions of $\Aut(\g)$, see \eg\ \cite[ch.\,4]{Onishchik2004_real}. The second statement can be checked using Gantmacher's normal form for automorphisms, \cf \cite[Thm.\,4.2]{Onishchik2004_real}.
\end{proof}
We will also need the following Lemma, which lets us compare the invariant ring of a quasi-compact real form with that of any real form inner to it:
\begin{lem}
\label{lem:comparison_map}
    Let $\sigma$ be a real structure of a connected complex reductive Lie group $G$ with Lie algebra $\g$. There exist a quasi-compact real structure $\sigma_0$ and a compact real structure $\tau$ of $G$ such that
    \begin{itemize}
        \item $\sigma_0$ is inner to $\sigma$,
        \item $\sigma$ and $\sigma_0$ both commute with $\tau$,
        \item there exists a Cartan subalgebra $\h\le \g$ preserved by $\sigma$, $\sigma_0$ and $\tau$, and
        \item $\h^\sigma = \h^{\sigma_0}$.
    \end{itemize}
    Moreover, let $\theta\coloneqq \sigma\tau$ and $\theta_0\coloneqq \sigma_0\tau$ denote the Cartan involutions, and $W(\theta)$, $W(\theta_0)$ the Weyl groups for $(\g^\theta, \h^\theta)$ and $(\g^{\theta_0}, \h^{\theta_0})$.\footnote{Here we are using Lemma \ref{lem:fixed_subalgebras} to see that $\h^\theta$ and $\h^{\theta_0}$ are Cartan subalgebras. Also, $W(\theta_0) = W^{\theta_0}$ is the centraliser of $\theta_0$ in $W$ by Lemma \ref{lem:fixed_subalgebras}(iii), but such a description need not hold for $W(\theta)$.} Then, as subgroups of $\GL(\h^\theta) = \GL(\h^{\theta_0})$, we have
    $$
        W(\theta) \le W(\theta_0),
    $$
    and it follows that $H^{2*}_{G^{\sigma_0}}$ canonically injects into $H^{2*}_{G^\sigma}$.
\end{lem}
\begin{proof}
    The existence of $\sigma_0$ is rather standard: the first requirement can be achieved via Proposition \ref{prop:qc_funfacts}, and the second via a variant of \cite[Prop.\,3.7]{Onishchik2004_real}. The third and fourth conditions can be incorporated by conjugation with an appropriate inner automorphism. For the inclusion of Weyl groups, it suffices to check that each element of the root system of $(\g^\theta, \h^\theta)$ is -- up to nonzero rescaling -- contained in that of $(\g^{\theta_0}, \h^{\theta_0})$; this can be done using Gantmacher normal forms, \cf \cite[Thm.\,4.2]{Onishchik2004_real}. 
    
    Finally, via \eqref{eq:long_iso_base_level}--\eqref{eq:killing_ring_isos} and Corollary \ref{cor:crt_extended} we obtain a diagram
    $$
    \begin{tikzcd}
        H^{2*}_{G^{\sigma_0}}
            \arrow[d, "\cong"]
            \arrow[r, dashrightarrow]
        & H^{2*}_{G^\sigma}
            \arrow[d, "\cong"] \\
        \C[\g^{\theta_0}]^{G^{\theta_0}}
            \arrow[r, dashrightarrow]
            \arrow[d, "\cong"]
        & \C[\g^{\theta}]^{G^{\theta}}
            \arrow[d, "\cong"] \\
        \C[\h^\theta]^{N_{G^{\theta_0}}(\h^\theta)}
           \arrow[r, dashrightarrow]
        & \C[\h^\theta]^{N_{G^{\theta}}(\h^\theta)}
    \end{tikzcd}
    $$
    in which the dashed arrows are yet to be constructed -- of course it suffices to construct one of them. The rings in the bottom row are the subrings of $\C[\h^\theta]^{W(\theta_0)}$ resp.\;$\C[\h^\theta]^{W(\theta)}$ invariant under the actions of the relevant component groups. But, by the same argument as in the proof of Corollary \ref{cor:disconn_inv_ring_nice}, we see that these component groups both act trivially on $\C[\h^\theta]^{W(\theta_0)}$. Together with the containment $W(\theta)\le W(\theta_0)$, this lets us put a canonical injection in the bottom row of the diagram above, finishing the proof. 
\end{proof}

\section{Proof of main theorem}
\label{sec:main_proof}
We now come to the proof of Theorem \ref{thm:main}. It is structured into three lemmas, followed by a main body combining everything. The first two lemmas reduce from the semisimple to the simple case; the third gives a construction of the appropriate real structure in almost all simple types. (For the one remaining type, $A_{2n}$, we supply ad hoc arguments.)
\begin{lem}
\label{lem:kirillov_decomp}
    Let $\g$ be a complex semisimple Lie algebra, and let $\g\cong \g_1 \oplus \cdots \oplus \g_\ell$ be its decomposition into simple complex Lie algebras $\g_i$. Under this isomorphism, a minuscule weight $\lambda$ of $\g$ decomposes as a sum $\lambda = \lambda_1 \oplus \cdots \oplus \lambda_\ell$ of minuscule weights of the $\g_i$. Moreover, the Kirillov algebra $\Kir^\lambda(\g)$ decomposes as
    $$
        \Kir^\lambda(\g)
        \cong
        \Kir^{\lambda_1}(\g_1) 
        \otimes \cdots \otimes
        \Kir^{\lambda_\ell}(\g_\ell).
    $$
\end{lem}
\begin{proof}
    The decomposition of $\lambda$ is a simple consequence of highest weight theory, which also yields $V^\lambda \cong V^{\lambda_1}\otimes \cdots \otimes V^{\lambda_\ell}$ (with $\g_i$ acting on the $i$-th tensor factor). It follows that both factors in $S(\g)\otimes \End(V^\lambda)$ decompose as tensor products compatibly with the decomposition of $\g$; hence, so does the Kirillov algebra.
\end{proof}

\begin{lem}
\label{lem:realstr_decomp}
    Let $G$ be a connected complex semisimple Lie group of adjoint type\footnote{This ensures that $G$ decomposes into simple factors.}, and let $G\cong G_1 \times \cdots \times G_\ell$ be its decomposition into simple factors $G_i$. Let $\sigma$ be a real structure on $G$. Then $\sigma$ permutes the $G_i$ with orbits of one or two elements. For each $i$, there are two possibilities:
    \begin{itemize}
        \item[(a)] If $\sigma(G_i) = G_i$, then $\sigma$ restricts to a real structure of $G_i$.
        \item[(b)] If $\sigma(G_i) = G_j$ with $i\ne j$, then $G_j\cong G_i$ and the Cartan involution of $\sigma|_{G_i\times G_j}$ is isomorphic to the \emph{swap involution}
        $$
            G_i\times G_i \to G_i \times G_i,
            \quad
            (g,h) \mapsto (h,g).
        $$
    \end{itemize}
\end{lem}
\begin{proof}
    For all $i,j$, the intersection $\sigma(G_i)\cap G_j$ is either trivial or all of $G_j$ by simplicity, so $\sigma$ indeed permutes the factors. The statement about orbit sizes and part (a) are clear. For part (b), choose any compact real structure $\tau_i$ of $G_i$ and observe that $\sigma\tau_i\sigma$ is a compact real structure of $G_j$ -- indeed, its fixed point subgroup is isomorphic via $\sigma$ to the compact Lie group $G_i^{\tau_i}$. The product $\tau \coloneqq \tau_i\times (\sigma\tau_i\sigma)$ is then a compact real structure of $G_i\times G_j$ which commutes with $\sigma$. By construction, the Cartan involution $\theta\coloneqq \sigma \tau \in \Aut_2(G_i\times G_j)$ maps $G_i$ to $G_j$, and the resulting $G_i\cong G_j$ can be used to identify $\theta$ with the swap involution up to isomorphism.
\end{proof}

\begin{lem}
\label{lem:qcrf_extension}
    Let $\g$ be a simple complex Lie algebra with a minuscule coweight $\lambda$, and let $\el_\lambda$ be the corresponding Levi subalgebra (\ie\ the centraliser of $\lambda$). Let $\mathfrak{S}$ be an inner class of real structures on $\g$ which preserve the $\Ad(\g)$-orbit of $\lambda$. If $\g$ is not of type $A_{2n}$ ($n\in \N$), then $\mathfrak{S}$ contains a real structure $\sigma$, unique up to inner-isomorphism, for which
    \begin{itemize}
        \item[(i)] $\sigma_*\lambda = \lambda$, and
        \item[(ii)] $\sigma|_{\el_\lambda}$ is quasi-compact.
    \end{itemize}
\end{lem}
\begin{proof}
    We may identify $\lambda$ with the fundamental coweight of a simple root $\alpha_1$ with respect to a Borel subalgebra $\b\le \g$ and Cartan subalgebra $\h\le \b$. Denote the remaining simple roots by $\alpha_2,\ldots, \alpha_r$ and extend this data to a pinning by choosing root vectors $e_i\in \g_{\alpha_i}$. These choices yield a set of standard generators $e_i,f_i,h_i$, which define a canonical Chevalley involution (also known as Weyl involution) $\omega\colon \g\to \g$ with
    $$
        \omega(h_i) = -h_i,
        \quad
        \omega(e_i) = - f_i,
        \quad
        \omega(f_i) = - e_i.
    $$
    as well as a canonical split real structure $\sigma_s$ which fixes the generators. Their composition $\tau\coloneqq \sigma\omega$ is a canonical compact real structure. (For details see \cite[p.\,18]{Onishchik2004_real}.)

    With this setup, $\el_\lambda$ inherits the standard generators $e_i, f_i, h_i$, $i\ge 2$; in particular $\omega$ restricts to $\el_\lambda$ as its canonical Chevalley involution. Following \cite[ch.\,4]{Onishchik2004_real}, we explicitly relate $\omega|_{\el_\lambda}$ to a pinning-preserving involution in its inner class. Let $\mathbf{h} = \sum_{i=2}^r r_i h_i$ be the dual of the sum of positive roots of $\el_\lambda$. (Here the coefficients $r_i$ are certain natural numbers, listed for instance in \cite[p. 78, Table 4]{Onishchik2004_real}.) Extend it to an $\sl_2$-triple by $\mathbf{e}\coloneqq \sum_{i=2}^r \sqrt{r_i}e_i$ and $\mathbf{f}\coloneqq \sum_{i=2}^r \sqrt{r_i}f_i$. It is not hard to show that $\varphi \coloneqq \exp(\ad(\tfrac{\pi}{2}(\mathbf{e}-\mathbf{f}))) \in \Int(\el_\lambda)$ commutes with $\omega|_{\el_\lambda}$ and that $\omega|_{\el_\lambda} \varphi$ is a pinning-preserving involution of $\el_\lambda$. Being an inner automorphism, $\varphi$ has an obvious extension $\Tilde{\varphi}\in \Int(\g)$ which commutes with $\omega$ and $\tau$, and fixes $\lambda$. We claim that $\Tilde{\varphi}^2 = \id$.  

    Since $\varphi^2 = \id_{\el_\lambda}$, it suffices to show that $\Tilde{\varphi}^2(e_1) = e_1$. (Indeed, any automorphism of $\g$ is uniquely determined by its action on $\h$ and the $e_i$, see \eg\ \cite[Thm.\,14.2]{Humphreys1972_LART}.) To this end, observe that $e_1$ is a lowest weight vector for an $\sl_2$-invariant subspace $V$ of $\g$ with respect to the $\sl_2$-triple $(\mathbf{e}, \mathbf{f}, \mathbf{h})$. Lifting the representation to $\SL_2$ identifies $\Tilde{\varphi}^2$ with the action of $\begin{psmallmatrix}
        -1 & 0\\
        0 & -1
    \end{psmallmatrix}$, so we have to show that the lowest weight of $V$ is even, or equivalently that $\dim V$ is odd. Since $\lambda$ is minuscule, the highest root $\alpha_{max}$ of $\g$ is of the form $\alpha_1 + \sum_{i=2}^r n_i\alpha_i$ for $n_i\in \N$. Moreover, it is clear that $\ad(\mathbf{e})^n (e_1) \in \g_{\alpha_{max}}$ for $n\coloneqq \sum_{i=2}^r n_i$. On the other hand, it is also clear that $\ad(\mathbf{f})^n$ maps $\g_{\alpha_{max}}$ to $\g_{\alpha_1}$. These observations combine to show that $\g_{\alpha_{max}}$ is contained in $V$ as its highest weight space, and that $\dim V = n + 1$ -- the height of $\alpha_{max}$. This number is known to equal $c-1$ where $c$ is the Coxeter number (see \eg\ \cite[Prop.\,VI.31]{Bourbaki_4to6}). Altogether, we have $\Tilde{\varphi}^2 = \id$ provided that the Coxeter number of $\g$ is even, which is guaranteed by the exclusion of type $A_{2n}$.

    If $\sigma_s\in \mathfrak{S}$, we may now take $\sigma\coloneqq \sigma_s\Tilde{\varphi}$, which indeed fixes $\lambda$ and whose Cartan involution $\omega \varphi$ preserves a pinning of $\el_\lambda$. In general, the assumption on $\mathfrak{S}$ implies that there exists an involution $\eta\in \Aut(\g)$, preserving the chosen pinning, fixing $\lambda$, and commuting with $\omega$ and $\sigma_s$, such that $\eta\sigma_s\in \mathfrak{S}$ (\cf\ proof of Proposition \ref{prop:sigma_action_translated}). Then $\eta$ commutes also with $\Tilde{\varphi}$ and we may take $\sigma\coloneqq \eta \sigma_s \Tilde{\varphi}$. This proves existence.

    As for uniqueness, suppose now that $\sigma'$ is any (other) real structure with the required properties. Up to conjugation by inner automorphisms, we may assume that $\sigma'$ preserves the previously chosen Cartan subalgebra $\h$ and commutes with $\tau$, and that the Cartan involution $\theta' \coloneqq \sigma'\tau$ preserves the chosen pinning of $\el_\lambda$. But since $\theta'$ is inner to $\theta\coloneqq\sigma\tau$, the same holds for their restrictions to $\el_\lambda$. Then, since both preserve the same pinning, we find $\theta'|_{\el_\lambda} = \theta|_{\el_\lambda}$. The inner automorphism $\psi\coloneqq \theta\theta'$ then normalises $\h$ and fixes all simple roots, whence $\psi = \exp(\ad t)$ for some $t\in \h$. Since $\psi$ also fixes $\el_\lambda$ pointwise, $t$ must be a scalar multiple of $\lambda\in \h$. But this yields
    $$
        \theta'
        =
        \theta \psi
        =
        \theta \exp (\ad t)
        =
        \exp(\ad (-t/2)) \,  \theta  \, \exp(\ad(t/2)),
    $$
    so $\theta'$ is inner-isomorphic to $\theta$ as we needed to show.
\end{proof}

\begin{rmk}[Satake diagrams]
\label{rmk:Satake_check}
    One can also verify the lemma using Satake diagrams. As before, we can identify $\lambda$ with the fundamental coweight of a simple root. We may then restrict to real structures $\sigma\in \mathfrak{S}$ which preserve $\h$ and for which $\h^\sigma$ is maximally split.\footnote{\ie\ $\dim(\h^\theta)$ is minimal possible where $\theta$ denotes the Cartan involution.} These are classified up to inner-isomorphism by Satake diagrams \cite{Araki1962_classification}, which consist of the Dynkin diagram of $\g$ together with a choice of painted nodes and an involutive permutation of the unpainted nodes, indicated by arrows. 
    
    The assumption that the elements of $\mathfrak{S}$ preserve the $\Ad(\g)$-orbit of $\lambda$ implies that the node representing $\lambda$ has no arrow attached (in any of the Satake diagrams resulting from $\mathfrak{S}$). One then verifies that condition (i) is equivalent to that node being unpainted. Moreover, deleting an unpainted node with no arrow from a Satake diagram yields the Satake diagram for the restricted real structure on the corresponding Levi subalgebra. Thus, it suffices to check that there is exactly one Satake diagram corresponding to a real structure in $\mathfrak{S}$ such that the node for $\lambda$ is unpainted and such that deleting that node yields the Satake diagram of a quasi-compact real structure of $\el_\lambda$. This is indeed the case whenever $\g$ is not of type $A_{2n}$, as can be verified using the tables in Appendix \ref{sec:uniqueness_and_tables}. The advantage of this approach is that it provides the inner-isomorphism type of $\sigma$ via its Satake diagram.
\end{rmk}

\begin{proof}[Proof of Theorem \ref{thm:main}]
    Since $G$ is simply connected, its Langlands dual $G^\vee$ is of adjoint type and Lemmas \ref{lem:kirillov_decomp} and \ref{lem:realstr_decomp} apply. Clearly, the permutation of simple factors of $G^\vee$ is the same for all $\sigma\in \mathfrak{S}$. By restricting to its orbits, the theorem is reduced to three cases:
    \begin{enumerate}
        \item[(a)] $\g$ is simple and not of type $A_{2n}$
        \item[(b)] $\g\cong \sl_{2n+1}$ for some $n\in \N$.
        \item[(c)] $\g\cong \g_s \oplus \g_s$ for $\g_s$ simple, and the real structures in $\mathfrak{S}$ swap the two copies of $\g_s$.
    \end{enumerate}
    In each case, we now give a real structure $\sigma\in \mathfrak{S}$ for which $\sigma_*\lambda = \lambda$ and such that the identity \eqref{eq:main_1} holds. The latter is equivalent to
    \begin{equation}
    \label{eq:main_proof_part_1}
        (H^{2*}_{L_\lambda})_{\sigma^*}
        \cong
        H^{2*}_{L_\lambda^\sigma}
    \end{equation}
    by Corollary \ref{cor:geom_model} and \eqref{eq:eq_coh_of_hmg_sp}.
    \begin{itemize}
        \item In case (a), we apply Lemma \ref{lem:qcrf_extension} to obtain $\sigma\in \mathfrak{S}$ fixing $\lambda$ with quasi-compact restriction to $L_\lambda$. Then \eqref{eq:main_proof_part_1} follows from Corollary \ref{cor:qc_main}.
        \item In case (b), we have $G^\vee = \PGL_{2n+1}(\C)$. Only the inner class of the split real structure fixes any minuscule coweights, and this inner class consists of a single inner-isomorphism class. We therefore only have to analyse the real structure $\sigma\colon A \mapsto \overline{A}$. The minuscule coweights of $\PGL_{2n+1}$ are exactly the fundamental coweights (and zero), and are (for standard choices) all fixed by $\sigma$. The corresponding Levi subgroups are $L_k\coloneqq \mathrm{P}(\GL_k\times \GL_{2n+1-k})$ for $k\in \N$, and using \eqref{eq:long_iso_base_level} and Corollary \ref{cor:crt_extended} one finds that
        $$
            H^{2*}_{L_k}
            \cong
            \C[x_1,\ldots, x_k, y_1,\ldots, y_{2n+1-k}]^{S_k\times S_{2n+1-k}}/(x_1 + \cdots + y_{2n+1-k}).
        $$
        
        The Cartan involution $A\mapsto (A^t)\inv$ acts as $-1$ on a Cartan subalgebra, so Lemma \ref{lem:from_eq_coh_to_invrings} and Theorem \ref{thm:CRT} imply that $\sigma^*$ acts on homogeneous elements by multiplication with $(-1)^{\deg}$. The coinvariant ring is then isomorphic to
        $$
            \C[x_1,\ldots, x_k, y_1,\ldots, y_{2n+1-k}]^{S_k\times S_{2n+1-k}
            \times \Z_2^{2n+1}},
        $$
        with $\Z_2^{2n+1}$ acting by sign changes on the variables. But a computation similar to that in Example \ref{expl:typeAexpl} identifies that ring with $H^{2*}_{L_k^\sigma}$, proving \eqref{eq:main_proof_part_1} for this case.
        
        \item In case (c), all elements of $\mathfrak{S}$ are quasi-compact, and can be conjugated by an inner automorphism to fix $\lambda$. The restriction to $\el_\lambda$ is then also quasi-compact, so we can proceed as in case (a).
    \end{itemize}

    This concludes the case distinction; the rest of the proof is again uniform. Firstly, \eqref{eq:main_2} holds for any quasi-compact $\sigma_0\in \mathfrak{S}$ by Lemma \ref{lem:inner_class_only}, \eqref{eq:long_iso_base_level}, and Corollary \ref{cor:qc_main}. The injection \eqref{eq:main_3} is achieved by Lemma \ref{lem:comparison_map}. The maps in diagram \eqref{eq:main_diag} are all derived from restriction maps of invariant polynomial rings, so the diagram commutes.
\end{proof}

\section{Characterisation of freeness}
\label{sec:ff}
In this section, Theorem \ref{thm:main} is used to characterise freeness of the coinvariant homomorphism \eqref{eq:coinv_alg_hom}, resulting in a proof of Theorem \ref{thm:degeneracy}. According to Theorem \ref{thm:main}, we have to analyse the composition
\begin{equation}
\label{eq:composition_to_be_analysed}
    H^{2*}_{(G^\vee)^{\sigma_0}}
    \overset{\varphi}{\rightarrow}
    H^{2*}_{(G^\vee)^{\sigma}}
    \rightarrow
    H^{2*}_{(G^\vee)^{\sigma}}((G^\vee)^\sigma/P_\lambda^\sigma)
\end{equation}
where $\varphi$ is the canonical injection constructed in Lemma \ref{lem:comparison_map}. 
\begin{lem}
\label{lem:phi_finite}
    $\varphi$ is finite.
\end{lem}
\begin{proof}
    Indeed, it is an injection between finitely generated $\C$-algebras of equal transcendence degree (namely $\operatorname{rank}(G^\vee)^\sigma = \operatorname{rank}(G^\vee)^{\sigma_0}$).
\end{proof}
The behaviour of the second map in \eqref{eq:composition_to_be_analysed} is related to the geometry of the homogeneous space $X\coloneqq (G^\vee)^\sigma/P_\lambda^\sigma$. Here it is more convenient to work with compact Lie groups, so we fix a compact real structure $\tau$ of $G^\vee$ that commutes with $\sigma$. It will be convenient to choose $\tau$ such that $\tau_*\lambda = -\lambda$, which can be achieved by a standard construction of compact real structures, \cf\cite[Thm.\,6.11]{KnappLG}. Now, $K\coloneqq ((G^\vee)^\sigma)^\tau$ is a maximal compact subgroup of $(G^\vee)^\sigma$. Since $(G^\vee)^\sigma \cong KP_\lambda^\sigma$ \cite[Prop.\,7.83f]{KnappLG}, $K$ acts transitively on $X$, so $X\cong K/L$ where $L\coloneqq P_\lambda^\sigma \cap K= L_\lambda^\sigma \cap K$. The cohomology of such compact homogeneous spaces is particularly well-behaved in the so-called \emph{equal rank} case:
\begin{thm}
\label{thm:equal_rank}
    If $K$ is a connected compact Lie group and $L\le K$ a closed subgroup, then the odd singular cohomology of $X=K/L$ vanishes if and only if $K$ and $L$ have the same rank. In this case, we further have an isomorphism $H^{2*}_K(X) \cong H^{2*}(X)\otimes_\C H^{2*}_K$ of $H^{2*}_K$-modules.
\end{thm}
\noindent
\textbf{Note:} In our convention, \emph{rank} means dimension of maximal torus. This need not coincide with the rank of the root system (due to a possibly positive-dimensional centre).
\begin{proof}
    The first statement is classical: the ``only if'' part follows from vanishing of the Euler characteristic $\chi(X)$ in the nonequal rank case first shown by Hopf and Samelson \cite{HopfSamelson1941}. The ``if'' part is due to Borel \cite{Borel1953_esp_fib_et_hom}. For a more detailed discussion, see \cite[ch.\,5]{Carlson_thesis}. The statement about equivariant cohomology (known as \emph{equivariant formality}) follows from degeneracy of the Serre spectral sequence for the fibration $X_K\to \mathbb{B}K$ with fibre $X$, \cf \cite[ch.\,9]{Carlson_thesis}.
\end{proof}
Returning to the analysis before Theorem \ref{thm:equal_rank}, we are led to compare the ranks of $K=((G^\vee)^\sigma)^\tau$ and $L=L_\lambda^\sigma\cap K$. We translate the question to the complex setting via the Cartan involution $\theta\coloneqq \sigma \tau$, noting that $K$ and $L$ are compact real forms of $(G^\vee)^\theta$ and $L_\lambda^\theta$, respectively. The theorem then implies:
\begin{lem}
\label{lem:eq_rank_is_good}
    If $(G^\vee)^\theta$ and $L_\lambda^\theta$ have the same rank, then 
    $
        H^{2*}_{(G^\vee)^{\sigma}}((G^\vee)^\sigma/P_\lambda^\sigma)
    $
     is a free module over $H^{2*}_{(G^\vee)^{\sigma}}$. Otherwise, $H^{2*}_{(G^\vee)^{\sigma}}$ has strictly larger transcendence degree than $H^{2*}_{(G^\vee)^{\sigma}}((G^\vee)^\sigma/P_\lambda^\sigma)$, so the module structure cannot be free.
\end{lem}
\begin{proof}
    The first assertion indeed follows immediately from Theorem \ref{thm:equal_rank}. For the statement about transcendence degrees, let $\h\le (\g^\vee)^\theta$ and $\h'\le \el_\lambda^\theta$ be Cartan subalgebras. By Corollary \ref{cor:crt_extended}, the rings involved are invariant subrings of $\C[\h]$ and $\C[\h']$ by finite groups.\footnote{Finiteness follows from the well-known finiteness of Weyl groups and the fact that $G^\theta$ and $L_\lambda^\theta$ have finitely many connected components. One way to see this is that they are homotopic to their intersections with the compact group $K$.} Their transcendence degrees are then the dimensions of $\h$ and $\h'$, respectively. If $H^{2*}_{(G^\vee)^{\sigma}}$ has strictly larger transcendence degree than $H^{2*}_{(G^\vee)^{\sigma}}((G^\vee)^\sigma/P_\lambda^\sigma)$, the structure map $H^{2*}_{(G^\vee)^{\sigma}} \to H^{2*}_{(G^\vee)^{\sigma}}((G^\vee)^\sigma/P_\lambda^\sigma)$ is not injective, so it does not define a free module structure.
\end{proof}
We are thus led to compare Cartan subalgebras of $(\g^\vee)^\theta$ and $\el_\lambda^\theta$, with the following result:
\begin{lem}
\label{lem:rank_comparison}
    Let $\g$ be a complex reductive Lie algebra and $\theta\in \Aut_2(\g)$. Among Cartan subalgebras of $\g$ preserved by $\theta$, choose $\h$ such that $\dim \h^\theta$ is minimal. Let $\lambda$ be a minuscule coweight of $\g$ contained in $\h^{-\theta}$ (\ie\ $\theta(\lambda) = - \lambda$) and $\el_\lambda$ the corresponding Levi subalgebra of $\g$. If $\theta$ preserves a pinning of $\el_\lambda$, then $\g^\theta$ and $\el_\lambda^\theta$ have equal rank if and only if $\theta$ is also pinning-preserving for $\g$.
\end{lem}
\begin{proof}
    Since $\el_\lambda$ contains the centre of $\g$, we may assume that $\g$ is semisimple. The ranks of $\g^\theta$ and $\el_\lambda^\theta$ are the maximal dimensions of $\te^\theta$ for $\theta$-stable Cartan subalgebras $\te$ of $\g$ or of $\el_\lambda$, respectively (\cf\ Lemma \ref{lem:fixed_subalgebras}). In particular, both ranks admit the lower bound $\dim \h^\theta$. 
    
    For convenience, we now translate to the parallel setting of real structures, which is more easily found in the literature. That is, we fix a compatible compact real structure $\tau$ of $\g$ and define $\sigma\coloneqq \tau\theta = \theta\tau$. Then $\te$ as above correspond to \emph{maximally compact} Cartan subalgebras of $\g^\sigma$ or $\el_\lambda^\sigma$. 
    
    The conjugacy classes of Cartan subalgebras of a real reductive Lie algebra can be related by \emph{Cayley transforms}, which are defined using \emph{real or noncompact imaginary roots}, \cf\cite[p.\,390]{KnappLG}. In particular, if there are no such roots with respect to $\sigma$, then there is only one such conjugacy class. If the Cartan involution $\theta$ preserves a pinning $(\te,\Pi,\{X_\alpha\})$, there are no real roots with respect to $\te$ and no noncompact imaginary simple roots in $\Pi$. Now suppose that, additionally, $\theta^*\alpha$ is either equal to or orthogonal to $\alpha$ for every simple $\alpha$. Using \cite[Prop.\,6.104]{KnappLG}, this implies that there are no real or noncompact imaginary roots at all, so all Cartan subalgebras of of the real form are conjugate. Since $\theta$ acts on the simple roots according to an automorphism of the Dynkin diagram, this assumption holds in all simple types except for the nontrivial involution of type $A_{2n}$.

    Thus, suppose that $\el_\lambda$ has no simple summand of type $A_{2n}$ on which $\theta$ restricts to a nontrivial involution. Then the above analysis implies that $\el_\lambda^\sigma$ has only one conjugacy class of $\theta$-stable Cartan subalgebras. In particular $\h^\theta$ must be a Cartan subalgebra of $\el_\lambda^\theta$, and so $\operatorname{rank}\el_\lambda^\theta = \dim \h^\theta$. A similar analysis shows that if $\theta$ is pinning-preserving on $\g$ and there are no interfering type $A_{2n}$ summands, $\g^\theta$ has rank $\dim\h^\theta$ as well. To finish the proof we thus have to do the following:
    \begin{enumerate}
        \item Show that a non-quasi-compact real form of a complex reductive Lie algebra always admits more than one isomorphism class of Cartan subalgebras. Since only one of them is maximally compact \cite[Prop.\,6.61]{KnappLG}, it will follow that $\operatorname{rank} \g^\theta > \dim\h^\theta$ in this case.
        \item Check the cases involving type $A_{2n}$ summands with non-inner Cartan involutions separately.
    \end{enumerate}
    For task (1), it suffices to show that such a real form always has noncompact imaginary roots with respect to a maximally compact Cartan subalgebra. This follows from the classification using \emph{Vogan diagrams}, \cf\cite[VI.8]{KnappLG}, since the absence of noncompact imaginary roots would imply quasi-compactness (by uniqueness of the classification).

    For task (2), we may assume that $\g$ is simple. For $\g=\sl_{2n+1}$, one finds that no non-inner involution restricts to a pinning preserving involution on any $\el_\lambda$ defined as above. Thus, we only have to treat the cases where $\el_\lambda$ has some type $A_{2n}$ summand on which the Cartan involution acts nontrivially. Up to isomorphism, they are as follows:
    \begin{itemize}
        \item $\g= \sl_{2n}(\C)$ with $\el_\lambda = \s(\gl_{2k+1}\oplus \gl_{2n-2k-1})(\C)$ for $k,n\in \N$; $\g^\theta = \so_{2n}(\C)$, and $\el_\lambda^\theta = \so_{2k+1}(\C) \oplus \so_{2n-2k-1}(\C)$. 
        Here $\theta$ is not pinning-preserving on $\g$ and the fixed subalgebras have ranks $n$ and $n-1$, respectively.
        \item $\g = \sp_{4n+2}(\C)$ for $n\in \N$ with $\el_\lambda = \sl_{2n+1}(\C)$; $\g^\theta = \gl_{2n+1}(\C)$ and $\el_\lambda^\theta = \so_{2n+1}(\C)$.
        Here, $\theta$ is again not pinning-preserving on $\g$, and the fixed subalgebras have ranks $2n+1$ and $n$, respectively.
    \end{itemize}
\end{proof}
\begin{rmk}
    It seems likely that Lemma \ref{lem:rank_comparison} can be proven more systematically and with fewer assumptions, perhaps using qualitative properties of Cayley transforms, \cf\cite[VI.7]{KnappLG}.
\end{rmk}

We now collect the auxiliary results of this section into a proof of Theorem \ref{thm:degeneracy}, which characterises freeness of the coinvariant homomorphism \eqref{eq:coinv_alg_hom}.
\begin{proof}[Proof of Theorem \ref{thm:degeneracy}]
    For $\g=\sl_{2n+1}(\C)$, $n\in \N$, with $\mathfrak{S}$ containing a split real structure, the Theorem can be checked directly. Here all elements of the inner class are quasi-compact, and a simple computation (as in the proof of Theorem \ref{thm:main}) shows that the coinvariant homomorphism \eqref{eq:coinv_alg_hom} is indeed free in this case. For the remainder of the proof, we assume that no element of $\mathfrak{S}$ restricts to a split real structure on a type $A_{2n}$ factor of $G^\vee$. 

    It is then clear from the proof of Theorem \ref{thm:main} that we may choose $\sigma = \sigma_0$ whenever a quasi-compact real structure in the given inner class fixes $\lambda$. If this is the case, the first map $\varphi$ in \eqref{eq:composition_to_be_analysed} is an identity. Moreover, the second map is then free (\ie\ defines a free module structure) by Lemmas \ref{lem:rank_comparison} and \ref{lem:eq_rank_is_good}.

    On the other hand, if $\sigma$ is not quasi-compact, then Lemmas \ref{lem:rank_comparison} and \ref{lem:eq_rank_is_good} imply that the transcendence degree drops in the second step of \eqref{eq:composition_to_be_analysed}. By Lemma \ref{lem:phi_finite}, it stays the same in the first step, so overall the coinvariant homomorphism decreases transcendence degree. In particular, it is not injective, hence cannot give rise to a free module structure.
\end{proof}

\section{Action on fibres and weights}
\label{sec:weight_action}
In this section, we relate the involutions discussed above to combinatorial information about the representation $V^\lambda$. In the general case, where $\lambda$ is any dominant integral weight, one should work with the big algebra \cite{Hausel2024_avatars}, whose fibres over suitable points of $\Spec S(\g)^\g$ can be identified with the canonical basis of $V^\lambda$ (\cf \cite[final slide]{Hausel2024_StonyBrook}, see also \cite{HKRW2020_crystals}). For simplicity, we instead (continue to) restrict ourselves to the case where $\Kir^\lambda(\g)$ is commutative and thus agrees with the big algebra (\cf \cite[Thm.2.1]{Hausel2024_avatars}).
By Proposition \ref{prop:kirillov_facts}, this means that $V^\lambda$ is weight multiplicity free, so instead of the canonical basis we can work with just the set of weights $\wt(\lambda)$. We begin this section by describing how to identify $\wt(\lambda)$ with fibres of the map
\begin{equation}
\label{eq:Kirillov_Spec_projection}
    \pi \colon
    \Spec\Kir^\lambda(\g)
    \to
    \Spec S(\g)^\g.
\end{equation}

For involutions induced by real structures, we then describe how the fixed points of the action on weights are encoded by the coinvariant homomorphism \eqref{eq:coinv_alg_hom}. This is applied to the special case of split real structures, where we recover (the weight multiplicity free case of) a $q=-1$ phenomenon of Stembridge \cite{Stembridge1994_minuscule, Stembridge1996_canonical}. Along the way, we obtain a combinatorial obstruction to freeness of the coinvariant homomorphism \eqref{eq:coinv_alg_hom}, see Proposition \ref{prop:freeness_obstruction}. In all of this section, we \textbf{assume that $V^\lambda$ is weight multiplicity free}. For Theorem \ref{thm:fixed_point_count}, proved below, $\lambda$ is assumed minuscule.

The starting point is to consider, for $x\in \g$, the evaluation homomorphism $S(\g)^\g\cong S(\g^*)^\g\to \C$, in which we first use the Killing form to view $S(\g)^\g$ as $\g$-invariant polynomials on $\g$ and then evaluate these on $\g$. Applying this to the first factor of $\Kir^\lambda(\g)$, we obtain a map
$$
    \ev_x\colon \Kir^\lambda(\g)
    = (S(\g)^\g\otimes \End(V^\lambda))^\g
    \to
    \End(V^\lambda).
$$
The image of $\ev_x$ is a subalgebra of $\End(V^\lambda)$ which we denote by
$$
    \Kir^\lambda_x(\g)
    \coloneqq
    \ev_x(\Kir^\lambda(\g)).
$$

Let us now assume that $x$ is a regular semisimple element. Recall that its centraliser $\h\coloneqq \mathfrak{c}_\g(x)$ is then a Cartan subalgebra of $\g$. From the definition of $\Kir^\lambda(\g)$, we see that $\ev_x$ then has values in $(\End (V^\lambda))^\h = \End_\h (V^\lambda)$, the endomorphisms of $V^\lambda$ which commute with the action of $\h$. For such $x$, $\Kir^\lambda_x(\g)$ turns out to describe a fibre of the map $\pi$ of \eqref{eq:Kirillov_Spec_projection}:
\begin{prop}
\label{prop:evaluation_onto}
    If $x\in \g$ is a regular semisimple element with corresponding Cartan subalgebra $\h \coloneqq \mathfrak{c}_\g(x)$, then $\Kir^\lambda_x(\g) = \End_\h(V^\lambda)$. Since $V^\lambda$ is weight multiplicity free, it follows that $\Kir^\lambda_x(\g) \cong \C^{\wt(\lambda)}$. Moreover, viewing $x$ as a closed point of $\Spec S(\g)^\g \cong \Spec\C[\g]^\g \cong \Spec\C[\h]^W$, we further have $\pi\inv(x) \cong \Spec (\Kir^\lambda_x(\g))$. Thus, $\pi\inv(x)$ is a zero-dimensional reduced scheme with underlying set $\wt(\lambda)$. 
\end{prop}
\begin{proof}
    As observed above, $\ev_x$ takes values in $\End_\h(V^\lambda)$, so the first assertion reduces to surjectivity of $\ev_x\colon \Kir^\lambda(\g) \to \End_\h (V^\lambda)$. This surjectivity follows from \cite[Thm.\,1.4]{Panyushev2004_endoj} (see also \cite{Panyushev2002_covarj}) and the fact that the $G$-orbit of $x$ is closed in $\g$. 
    
    For the second assertion, note that endomorphisms in $\End_\h(V^\lambda)$ have to send each weight space to itself. By weight multiplicity freeness, each weight space is one-dimensional, so such endomorphisms are diagonal with respect to the weight space decomposition.

    It remains to show that $\pi\inv(x)$ is isomorphic to the spectrum of $\Kir^\lambda_x(\g)$. Let $\m_x \subset \C[\g]^\g$ be the kernel of the evaluation homomorphism taking $f\in \C[\g]^\g$ to $f(x)\in \C$. Then, by definition, we have
    $$
        \pi\inv(x)
        =
        \Spec \big(
            \Kir^\lambda(\g) \otimes_{\C[\g]^\g} (\C[\g]^\g/\m_x) 
        \big)
        =
        \Spec \big(
            \Kir^\lambda(\g)/(\Kir^\lambda(\g)\cdot \m_x)
        \big),
    $$
    whereas 
    $$
        \Kir^\lambda_x(\g) = \Kir^\lambda(\g)/\ker(\ev_x).
    $$
    Since $(\Kir^\lambda(\g)\cdot \m_x)$ is clearly contained in $\ker(\ev_x)$, there is a canonical surjection 
    $$
        \Kir^\lambda(\g) \otimes_{\C[\g]^\g} (\C[\g]^\g/\m_x)
        \twoheadrightarrow
        \Kir^\lambda_x(\g).
    $$
    But these are finite dimensional $\C$-algebras, and
    $$
        \dim_\C \big(\Kir^\lambda(\g) \otimes_{\C[\g]^\g} (\C[\g]^\g/\m_x))
        =
        \operatorname{rk}_{\C[\g]^\g}(\Kir^\lambda(\g))
        =
        \dim_\C (\End_\h(V^\lambda))
        =
        \dim_\C \Kir^\lambda_x(\g)
    $$
    (for the second equality, see \cite[p.\,4, above Cor.\,1]{Kirillov2001_introduction}), so the surjection has to be an isomorphism.
\end{proof}

Thus, the map $\pi$ of \eqref{eq:Kirillov_Spec_projection} has reduced fibres over all regular semisimple points in the base. We will also require the base points to be fixed by the involutions under consideration, which can be arranged without breaking the good fibre behaviour. Indeed, if $\sigma^*$ is an involution of the algebra $S(\g)^\g\to \Kir^\lambda(\g)$ induced by a real structure (see Section \ref{sec:real_str}), then the action of $\sigma^*$ on the base ring $S(\g)^\g$ is by a complex linear involution $\theta$ of $\g$ (\cf\ Lemma \ref{lem:from_eq_coh_to_invrings}). Moreover, only the inner class of $\theta$ matters, so we may assume that $\theta$ preserves a pinning. As observed in the proof of Lemma \ref{lem:dist_coinv}, a pinning-preserving automorphism fixes a regular semisimple element of $\g$ and we conclude:
\begin{lem}
\label{lem:good_base_points}
    For an involution $\sigma^*$ induced by a real structure, there is a dense open subset $U$ of the fixed point scheme $(\Spec S(\g)^\g)^{\sigma^*}$ such that $\pi\inv (x) \cong \Spec \C^{\wt(\lambda)}$ for every $x\in U$.
\end{lem}
\begin{proof}
    We may take $U$ to be the intersection of $(\Spec S(\g)^\g)^{\sigma^*}$ with the open subset of $\Spec S(\g) \cong \g$ of regular semisimple elements. Then $U$ is open in $(\Spec S(\g)^\g)^{\sigma^*}$, hence dense open (by irreducibility) since it is nonempty by the discussion above.
\end{proof}

For $x$ as in the lemma, the involution $\sigma^*$ restricts to an involution of $\pi\inv (x)$, given by an involutive permutation of the underlying set $\wt(\lambda)$. By Proposition \ref{prop:sigma_action_translated}, $\sigma^*$ is of the form $\iota_{\eta, -1}$ for a certain $\eta\in \Aut_2(\g)$, allowing for a convenient description of the induced permutation:
\begin{prop}
\label{prop:involution_to_wt_permutation}
    Let $\eta\in \Aut_2(\g)$ such that $\eta^*\lambda = \lambda$ (with notation as in Proposition \ref{prop:iota_eta_zeta}(i)). By replacing $\eta$ with an involution in the same inner class if necessary, we may assume that there exists a regular semisimple $x\in \g$ such that $\eta(x) = -x$. The restriction of $\iota \coloneqq \iota_{\eta, -1}$ to the fibre $\pi\inv(x) \cong \wt (\lambda)$ is then given by
    $$
        \wt(\lambda)\to \wt(\lambda),
        \quad
        \mu \mapsto \mu \circ \eta.
    $$
\end{prop}
\begin{proof}
    To arrange $\eta(x) = - x$ for some regular semisimple $x$, take $\eta$ to be a \emph{principal involution} in the given inner class -- see \cite[Def.\,6.13 and Thm.\,6.14]{AV1992_projective}. The involution of $\Kir^\lambda_x(\g)$ induced by $\iota$ is then simply $\Conj_{A_{\eta,\lambda}}$, which is easily seen to act on the weights as claimed.
\end{proof}

Since these fibres are zero-dimensional and reduced, it is easy to count the fixed points:
\begin{lem}
\label{lem:fp_tr_coinv}
    Let $A$ be a finite-dimensional reduced commutative $\C$-algebra and $\iota$ an algebra automorphism of $A$, corresponding to an automorphism of $\Spec A$ also denoted $\iota$. Then
    $$
        \# (\Spec A)^\iota
        =
        \tr \iota
        =
        \dim_\C A_\iota,
    $$
    where $A_\iota$ denotes the coinvariant ring (see Proposition \ref{prop:coinv_ring}).
\end{lem}
\begin{proof}
    The assumptions imply that $X\coloneqq \Spec A$ is finite and reduced, so $A\cong \C^X$. In particular, $A$ has a $\C$-basis consisting of functions $\delta_x$ which are $1$ on $x\in X$ and zero on $X\setminus \{x\}$. Moreover, $\iota$ permutes these basis functions, so that
    $$
        \tr \iota
        =
        \# \{x\in X\mid \delta_x \circ \iota = \delta_x\}
        = \# (\Spec A)^\iota.
    $$
    The span of the fixed basis functions can be identified (as a vector space) with 
    $$
        \C^X/\big( \C\cdot\{\delta_x\mid  \delta_x \circ \iota \ne \delta_x \} \big)
        \cong
        \C^X/\big( \C^X\cdot\{\delta_x\mid  \delta_x \circ \iota \ne \delta_x \} \big)
        \cong
        A_\iota,
    $$
    which finishes the proof.
\end{proof}

Thus, the number of fixed points in a suitable fibre $\pi\inv (x)$ is given by $\dim_\C \Kir^\lambda_x(\g)_{\sigma^*}$. If the full coinvariant ring $\Kir^\lambda(\g)_{\sigma^*}$ is already free over $(S(\g)^\g)_{\sigma^*}$ then that dimension is equal to its rank. Note that in this case the corresponding morphism on spectra is surjective, so in particular $\pi\inv (x)^{\sigma^*} \ne \emptyset$. Conversely, if $\pi\inv (x)^{\sigma^*} = \emptyset$, then the same argument shows that the coinvariant homomorphism \eqref{eq:coinv_alg_hom} is not injective. We can phrase this as a simple obstruction to freeness, shedding some light on Theorem \ref{thm:degeneracy}:
\begin{prop}
\label{prop:freeness_obstruction}
    Let $\iota = \iota_{\eta, -1}$ be an involution of $\Kir^\lambda(\g)$ of the form given in Proposition \ref{prop:iota_eta_zeta}. Assume that $\eta$ has a regular semisimple eigenvector in $\g$ with eigenvalue $-1$. If the set
    $$
        \{\mu \in \wt(\lambda) \mid \mu = \mu \circ \eta \}
    $$
    is empty, then the coinvariant homomorphism $S(\g)^\g_\iota \to \Kir^\lambda(\g)_\iota$ is not injective, hence not free.
\end{prop}
\begin{proof}
    Follows directly from the discussion above and Proposition \ref{prop:involution_to_wt_permutation}.
\end{proof}
\begin{rmk}
    When $\g$ is simple, it is not hard describe the action of $\eta$ on weights explicitly (as is done for inner $\eta$ below). For instance, if $\eta$ is a Chevalley involution (meaning that $\eta$ acts as $-1$ on a Cartan subalgebra), then this action is simply $\mu\mapsto -\mu$. If $0$ is not a weight of $V^\lambda$, this corresponds to a non-free coinvariant homomorphism. Along these lines, one could prove combinatorially that the freeness condition in Theorem \ref{thm:degeneracy} is necessary.
\end{rmk} 

For the rest of this section, we \textbf{specialise to the case where $\sigma$ is a split real structure}. 

In this case, $\sigma^* = (-1)^{\deg}$ according to Proposition \ref{prop:sigma_action_translated}. To extract information on fixed points in fibres from this fact, let us recall (from Proposition \ref{prop:kirillov_facts}) that $\Kir^\lambda(\g)$ is finite-free over its graded subring $S(\g)^\g$. Thus, there is a graded $\C$-vector space $V$ such that $\Kir^\lambda(\g)\cong S(\g)^\g\otimes_\C V$ as a graded $S(\g)^\g$-module. Using the description by invariant polynomial rings in Section \ref{sec:invrings_and_qcrs} one can show \cite[p.\,11]{Panyushev2004_endoj} that $V=V^\lambda$ with principal grading.\footnote{Alternatively, one can use the geometric model in Theorem \ref{thm:geom_model} together with the geometric Satake equivalence \cite{MV} to conclude this for all big algebras, in particular for minuscule Kirillov algebras.} This grading is defined in terms of the weight spaces $V^\lambda_\mu$, $\mu\in \wt{\lambda}$. Namely, all weights of $V^\lambda$ are obtained from the lowest weight $-\lambda^*$ by successively adding simple roots; if $k$ such additions are needed to obtain $\mu$ we place $V^\lambda_\mu$ in degree $k$. If $\Delta$ denotes the root system and $\rho^\vee \coloneqq \tfrac{1}{2}\sum_{\alpha\in \Delta^+}\alpha^\vee$ the half sum of positive coroots, then $\langle\rho^\vee, \alpha\rangle = 1$ for all simple roots $\alpha$. Thus, the degree of $V^\lambda_\mu$ can be computed as $k=\langle\rho^\vee, \mu + \lambda^*\rangle = \langle\rho^\vee, \mu + \lambda\rangle$ and the Poincaré polynomial for this grading is the \emph{Dynkin polynomial}, \cf\cite[\S3]{Panyushev2004_endoj}
\begin{equation}
\label{eq:dynkinpoly}
    \D^\lambda(q) 
    =
    \sum_{\mu \in \wt{\lambda}} (\dim V^\lambda_\mu) q^{\langle \mu + \lambda, \rho^\vee \rangle}
    =
    \prod_{\alpha \in \Delta^+} \frac{1 - q^{\langle \rho +\lambda, \alpha^\vee \rangle}}{1 - q^{\langle \rho, \alpha^\vee \rangle}}.
\end{equation}

From the graded module isomorphism $\Kir^\lambda(\g) \cong S(\g)^\g \otimes_\C V^\lambda$, it follows that $\Kir^\lambda_x(\g) \cong \Kir^\lambda(\g) \otimes_{S(\g)^\g} \C \cong V^\lambda$ as a graded $\C$-vector space, for any $x\in \g$. If $x$ is moreover fixed by $\sigma^*$, it follows that $\sigma^*$ restricts to $\Kir^\lambda_x(\g)$ as $(-1)^{\deg}$. In particular, its trace is given by the Poincaré polynomial evaluated at $-1$. Via Lemma \ref{lem:fp_tr_coinv} we conclude:
\begin{prop}
\label{prop:q=-1}
    Let $\g$ be a complex semisimple Lie algebra, $\lambda$ a minuscule weight, $G$ the connected simply connected Lie group with Lie algebra $\g$, $G^\vee$ its Langlands dual, and $\sigma$ a split real structure of $G^\vee$. Then, given $x\in (\Spec S(\g)^\g)^{\sigma^*}$ with reduced fibre as in Lemma \ref{lem:good_base_points}, we have
    $$
        \# (\pi\inv (x))^{\sigma^*} = \D^\lambda(-1).
    $$
\end{prop}

To finish the analysis, we now want to compute the restriction of $\sigma^*$ to a good fibre $\Spec \Kir^\lambda_x(\g)$ again, but from the point of view of $\Kir^\lambda_x(\g) \cong \C^{\wt{\lambda}}$. In this picture, the grading appears less natural and the explicit permutation on the set of weights from Proposition \ref{prop:involution_to_wt_permutation} is preferable. Here, the involution $\eta$ appearing in the proposition will be inner and map a regular semisimple element $x$ to its negative. Modifying $\eta$ if necessary, we may assume that $x$ lies in the dominant Weyl chamber, so that the action of $\eta$ on weights is given by the longest element $w_0$ of $W$. We are now in position to synthesise a proof of Theorem \ref{thm:fixed_point_count} from the discussion in this section.

\begin{proof}[Proof of Theorem \ref{thm:fixed_point_count}]
    The first claim of the Theorem is that there is a dense open subset $U$ of the fixed point scheme $(\Spec S(\g)^\g)^{\sigma^*}$ over which the fibres of \eqref{eq:Kirillov_Spec_projection} are reduced; this is shown in Lemma \ref{lem:good_base_points}. 
    
    For such a base point $x$, Proposition \ref{prop:q=-1} shows that the number of fixed points in its fibre is $\D^\lambda(-1)$. By the discussion immediately above this proof, this is also the number of weights of $V^\lambda$ fixed by the longest Weyl group element.\footnote{There, the base point $x$ is further restricted to be an eigenvector of the involution $\eta$, and it is not obvious whether this works for all $x\in U$. However, the fixed point count in terms of $\D^\lambda$ is independent of that restriction and gives the same result in all fibres over points in $U$.}

    Finally, it is claimed that the fibre is non-empty if and only if $\lambda$ is fixed by a quasi-compact real structure inner to a split real structure. According to Theorem \ref{thm:degeneracy}, the latter condition is equivalent to injectivity of the coinvariant homomorphism $S(\g)^\g_{\sigma^*} \to \Kir^\lambda(\g)_{\sigma^*}$. In turn, this is equivalent to surjectivity of the map $(\Spec \Kir^\lambda(\g))^{\sigma^*}\to (\Spec S(\g)^\g)^{\sigma^*}$ (using Proposition \ref{prop:coinv_ring} and the fact that the map is finite, hence closed). But it is clear that this map is surjective if and only if its fibres over the dense open subset $U$ are nonempty.
\end{proof}

\begin{rmk}
    Instead of the route taken here, one can also observe that the principal grading on $V^\lambda$ corresponds to weight spaces for a suitable element of the adjoint group $\Ad (\g)$, and that this element is conjugate to a representative of $w_0$. This is the original strategy used in \cite{Stembridge1994_minuscule}. The inner automorphism $\eta$, arising in the discussion above, which maps a regular semisimple element to its negative is indeed a special involution in the sense of \cite[p.\,15]{Stembridge1994_minuscule}. However, the approach we have taken here seems more natural in the context of Kirillov algebras.
\end{rmk}

Finally, we present a slight extension of Theorem \ref{thm:fixed_point_count}, which makes some progress towards proving conjectures of Hausel stated in \cite[slide 7]{Hausel2023_ICMAT}. Following Hausel, we introduce the following variant of \eqref{eq:dynkinpoly}, a priori as a rational function:
\begin{equation}
\label{eq:dinkin_even}
   \D^\lambda_{\mathrm{ev}}(q)
    \coloneqq
    \prod_{\substack{
        \alpha \in \Delta^+\\
        \langle \rho +\lambda, \alpha^\vee \rangle \in 2\Z
    }}\!\!\!\!     
        \big(1 - q^{\langle \rho +\lambda, \alpha^\vee \rangle}\big)
    \prod_{\substack{
        \alpha \in \Delta^+\\
        \langle \rho, \alpha^\vee \rangle \in 2\Z
    }}\!\!\!     
        \big(1 - q^{\langle \rho, \alpha^\vee \rangle}\big)\inv.
\end{equation}
Two simple facts about this are as follows:
\begin{prop}
\label{prop:Dynkin_even_basics}
\leavevmode
\begin{itemize}
    \item[(i)] $\D^\lambda_{\mathrm{ev}}$ is a polynomial.
    \item[(ii)] $\D^\lambda_{\mathrm{ev}}(1) = \D^\lambda(-1)$.
\end{itemize}
\end{prop}
\begin{proof}
    Since $\D^\lambda$ is a fraction of products of cyclotomic polynomials by \eqref{eq:dynkinpoly}, it admits a factorisation
    $$
        \D^\lambda(q)
        =
        \frac{\prod_{z\in Z_1} (q-z)}{\prod_{z\in Z_2} (q-z)}
    $$
    where $Z_1,Z_2$ are multisets of roots of unity different from $1$. The multiplicity of each $z$ in $Z_1$ is at least the multiplicity in $Z_2$ because $\D^\lambda$ is a polynomial. By removing all powers of odd primitive roots of unity from both $Z_1$ and $Z_2$, we arrive at a factorisation of $\D^\lambda_{\mathrm{ev}}$, which is thus also a polynomial, proving (i). The definition of $\D^\lambda_{\mathrm{ev}}$ immediately yields that it is an even polynomial, so in particular $\D^\lambda_{\mathrm{ev}}(1) = \D^\lambda_{\mathrm{ev}}(-1)$. Part (ii) then follows from (i) and the fact that
    $$
        (1-q^{2k+1})/(1-q) = 1 + q + \cdots q^{2k}
    $$
    evaluates to $1$ at $q=-1$ for any $k\in \N$.
\end{proof}
\begin{thm}
    Let $\g$ be a complex semisimple Lie algebra and $\lambda$ a minuscule weight of $\g$. Further, let $G$ be the connected simply connected Lie group with Lie algebra $\g$, $G^\vee$ its Langlands dual, and $\sigma$ a split real structure of $G^\vee$. Then the coinvariant homomorphism $S(\g)^\g_{\sigma^*}\to \Kir^\lambda(\g)_{\sigma^*}$ is free if and only if $\D^\lambda_{\mathrm{ev}}(1) \ne 0$. If this is the case, $\D^\lambda_{\mathrm{ev}}$ is the Poincaré polynomial of $\Kir^\lambda(\g)_{\sigma^*}$ over $S(\g)^\g_{\sigma^*}$.
\end{thm}
\begin{proof}
    The characterisation of freeness follows from Theorems \ref{thm:degeneracy} and \ref{thm:fixed_point_count} and Proposition \ref{prop:Dynkin_even_basics}. In the free case, we see from Theorems \ref{thm:degeneracy} and \ref{thm:main} that the coinvariant homomorphism is the canonical map
    $$
        H^{2*}_{(G^\vee)^\sigma}
        \to
        H^{2*}_{(G^\vee)^\sigma}((G^\vee)^\sigma/P_\lambda^\sigma)
    $$
    for a quasi-compact real structure inner to $\sigma$. Using the identification $\sigma^* = (-1)^{\deg}$ and unravelling the proof of Theorem \ref{thm:main}, we may identify this map with
    $$
        \C[\g]^\g_{(-1)^{\deg}}
        \to
        \C[\el_\lambda]^{\el_\lambda}_{(-1)^{\deg}}.
    $$
    
    These are polynomial rings with generators graded in the even degrees of the Weyl groups $W$ and $W_\lambda$. Therefore, the Poincaré polynomial is given by
    $$
        P(q)
        \coloneqq
        \prod_{d \in D(W)\cap 2\Z}\!\!\!\!(1-q^d)
        \prod_{d \in D(W_\lambda)\cap 2\Z}\!\!\!\!(1-q^d)\inv
    $$
    where $D(W), D(W_\lambda)$ are the respective multisets of degrees. But the analogous expression
    $$
        \prod_{d \in D(W)}\!\!\!(1-q^d)
        \prod_{d \in D(W_\lambda)}\!\!\!(1-q^d)\inv
    $$
    equals $\D^\lambda$ (see \cite[Prop.\,3.7]{Panyushev2004_endoj}), so a similar argument as in the proof of Proposition \ref{prop:Dynkin_even_basics} shows that $P=\D^\lambda_{\mathrm{ev}}$.
\end{proof}

\section{Outlook}
\label{sec:outlook}
As mentioned in the introduction, we expect our main results to generalise to Hausel's 
\emph{big algebras} \cite{Hausel2024_avatars}. In general, big algebras are maximal commutative subalgebras of Kirillov algebras and need not coincide with their ambient Kirillov algebras. Nevertheless, in the minuscule case they do coincide, and the geometric model used throughout this paper generalises more readily to big algebras (cf.\,Theorem \ref{thm:geom_model}). It therefore seems appropriate in our context to consider big algebras as a generalisation of the minuscule Kirillov algebras studied in this paper. On the geometric side, this replaces equivariant cohomology of partial flag varieties with equivariant intersection cohomology of affine Schubert varieties. 

An added difficulty is that this description does not account for the ring structure of big algebras in general, since intersection cohomology need not have such a structure. Nevertheless, there still is a well-defined action of real structures, and we do expect the coinvariant ring to be modelled by a real affine Schubert variety similarly to Theorem \ref{thm:main}. Such a geometric model would likely also lead to a characterisation of freeness as in Theorem \ref{thm:degeneracy}.
However, the relatively elementary methods used here need to be adapted to treat the general case. Let us also note that another direction suggested by Theorem \ref{thm:geom_model} is to study the involutions on the centres of general big algebras.

In the special case of a split real structure, Theorem \ref{thm:fixed_point_count} of this paper recovers a minuscule $q=-1$ phenomenon due to Stembridge \cite{Stembridge1994_minuscule}. That work was generalised to arbitrary representations, where one counts fixed elements of the canonical basis instead of fixed weights, see \cite{Stembridge1996_canonical}. We expect the action of split real structures on general big algebras to recover this via a similar analysis as in Section \ref{sec:weight_action}. 

Another future direction is to describe the case of non-free coinvariant homomorphisms $S(\g)^\g_{\sigma^*} \to \Kir^\lambda(\g)_{\sigma^*}$ in more detail. It follows from Lemma \ref{lem:good_base_points} and the proof of Theorem \ref{thm:fixed_point_count} that the image of $\Spec \Kir^\lambda(\g)_{\sigma^*}$ in  $S(\g)^\g_{\sigma^*}$ must in this case be contained in the vanishing locus $V(D)$ of the discriminant. One can then ask for a maximal closed subscheme $Z$ of $S(\g)^\g_{\sigma^*}$ (contained in $V(D)$) such that $\Spec \Kir^\lambda(\g)_{\sigma^*}$ is finite-free over $Z$. The analysis in Section \ref{sec:ff} suggests that such a $Z$ can be obtained as the spectrum of $\C[\kay]^K$ for a suitable (compact) subgroup $K$ of $(G^\vee)^{\sigma_0}$ with Lie algebra $\kay$, but it is not yet clear to us how to define $K$ in a uniform way.

Let us also remark that there is a connection to moduli spaces of Higgs bundles. When $\g$ is of type $A$, the spectra of minuscule Kirillov algebras are isomorphic to \emph{very stable type $(1,\ldots, 1)$ upward flows} in the moduli space of Higgs bundles\footnote{of fixed rank and degree, over a smooth projective curve of genus at least two} \cite[\S4]{GH2024_even_very_stable}. Involutions (and more general automorphisms) on the whole moduli space, resembling the automorphisms $\iota_{\eta,\zeta}$ of Proposition \ref{prop:iota_eta_zeta}, have been studied in \cite{GPRa2019_involutions}. The Kirillov algebra involutions studied here should correspond to such global involutions of the moduli space restricted to suitable upward flows. The case of $\iota_{\id, -1}$ is discussed in \cite{GH2024_even_very_stable}. 

Finally, one can ask how natural our focus on real structures is for the material covered in this paper. Indeed, significant parts of our arguments proceed via the (holomorphic) Cartan involutions corresponding to the real structures. However, the geometric model in Theorem \ref{thm:main} does seem most natural when stated through real structures (as opposed to invariant polynomial rings for complex Lie groups). Currently, this provides only heuristic input towards describing the coinvariant homomorphism: namely, it is free precisely when the modelling real partial flag variety is an equal rank homogeneous space. It would be interesting to further incorporate the real geometry into the description.

\appendix

\section{Remarks on uniqueness and tables}
\label{sec:uniqueness_and_tables}
In the proof of Theorem \ref{thm:main} (see \S\ref{sec:main_proof}), three cases were distinguished. In most cases, the real structure guaranteed by the Theorem can be chosen to restrict quasi-compactly to the relevant Levi subgroup; it is then essentially unique (see Lemma \ref{lem:qcrf_extension}). This restriction property can fail for simple factors of type $A_{2n}$ with the inner class of the split real structure. This inner class consists of only one inner-isomorphism class, Thus, we could obtain a unique choice of $\sigma$ up to isomorphism by demanding quasi-compact restriction on all factors where this is possible. 

The resulting $\sigma$ are precisely those used in the proof of Theorem \ref{thm:main}. However, in the way it is stated, Theorem \ref{thm:main} could allow for several essentially different choices of $\sigma$. For the sake of completeness, we remark that this can indeed happen. Namely, consider the case of $G^\vee = Sp_{4n}(\C)$ with $L_\lambda \cong \GL_{2n}(\C)$ and the inner class containing a split real form of $G^\vee$. The real structure provided by Lemma \ref{lem:qcrf_extension} results in the real form $U^*(2n)$ of $\GL_{2n}(\C)$. However, a calculation as in Example \ref{expl:typeAexpl} shows that we could also have chosen the split real structure of $G^\vee$, resulting in the real form $\GL_{2n}(\R)$ of $\GL_{2n}(\C)$.

The real structures used in the proof of Theorem \ref{thm:main} are listed below, in terms of the corresponding real forms, for all minuscule weights of simple complex Lie algebras. This leaves (up to direct sums) the case of $\g^\vee = \g_s\oplus \g_s$ where $\g_s$ is simple and the real structure permutes the summands. However, all such real structures are isomorphic; up to isomorphism the Satake diagram consists of two copies of the Dynkin diagram of $\g_s$ with corresponding nodes connected by arrows.

Except when $\g$ is of type $D_4$, there are then one or two inner classes of real structures for each $\g$, so we can group them by whether they are inner to a split real structure. If we identify isomorphic (though not necessarily inner-isomorphic) real structures, then even for $D_4$ this results in only two cases. Thus, Table \ref{table:split} contains inner classes of split real structures, and Table \ref{table:non_split} the remaining ones. We also list the Levi subalgebras $\el_\lambda$ via their (more conveniently notated) derived subalgebras $\el_\lambda'$ and include Satake diagrams. For the role of Satake diagrams in this context, see Remark \ref{rmk:Satake_check}. 

\begin{table}[ht]
    \centering
    \begin{tabular}{|c|c|c|c|c|}
        \hline
        $\g^\vee$ & $\el_\lambda'$ & $(\g^\vee)^\sigma$ & $(\el_\lambda^\sigma)'$ & Satake diagram 
        \Tstrut\Bstrut
        \\
        \hline
            $\sl_{2n}(\C)$ 
            & \makecell[c]{$\sl_{2k}(\C) \,\oplus$ \Tstrut\\ $\sl_{2n-2k}(\C)$} 
            & \makecell[c]{$\su^*(2n)$\Tstrut\\ (quasi-compact)} 
            & \makecell[c]{$\su^*(2k) \,\oplus$\Tstrut\\ $\su^*(2n-2k)$}
            & \dynkin[rootradius = 0.06cm, labels = {1,2,3,$2k$, , $2n-1$}]A{*o*.x.o*}
            \Tstrut\Bstrut
        \\
        \hline
            $\sl_{2n}(\C)$ 
            & \makecell[c]{$\sl_{2k+1}(\C) \,\oplus$ \Tstrut\\ $\sl_{2n-2k-1}(\C)$} 
            & $\sl_{2n}(\R)$
            & \makecell[c]{$\sl_{2k+1}(\R) \,\oplus$ \Tstrut\\ $\sl_{2n-2k-1}(\R)$}
            & \dynkin[rootradius = 0.06cm, labels = {1,2,3,$2k+1$, , $2n-1$}]A{ooo.x.oo}
            \Tstrut\Bstrut
        \\
        \hline
            $\sl_{2n+1}(\C)$ 
            & \makecell[c]{$\sl_{k}(\C) \,\oplus$\Tstrut\\ $\sl_{2n+1-k}(\C)$}
            & \makecell[c]{$\sl_{2n+1}(\R)$ \Tstrut\\ (quasi-compact)}
            & \makecell[c]{$\sl_{k}(\R) \,\oplus$\Tstrut\\ $\sl_{2n+1-k}(\R)$}
            & \dynkin[rootradius = 0.06cm, labels = {1,2,3,$k$, , $2n$}]A{ooo.x.oo}
            \Tstrut\Bstrut
        \\
        \hline
            $\so_{2n+1}(\C)$ 
            & $\so_{2n-1}(\C)$ 
            & $\so_{1,2n}(\R)$
            & $\so_{2n-1}(\R)$
            & \dynkin[rootradius = 0.06cm, labels = {1,2,3, ,$n$}]B{x**.**}
            \Tstrut\Bstrut
        \\
        \hline
            $\sp_{4n}(\C)$ 
            & $\sl_{2n}(\C)$ 
            & $\sp_{2n, 2n}(\R)$
            & $\su^*(2n)$
            & \dynkin[rootradius = 0.06cm, labels = {1,2,3, , ,$2n$}]C{*o*.o*x}
            \Tstrut\Bstrut
        \\
        \hline
            $\sp_{4n+2}(\C)$ 
            & $\sl_{2n+1}(\C)$ 
            & $\sp_{4n+2}(\R)$
            & $\sl_{2n+1}(\R)$
            & \dynkin[rootradius = 0.06cm, labels = {1,2,3, , ,$2n+1$}]C{ooo.oox}
            \Tstrut\Bstrut
        \\
        \hline
            $\so_{4n}(\C)$ 
            & $\so_{4n-2}(\C)$
            & $\so_{2,4n-2}(\R)$
            & $\so_{1, 4n-3}(\R)$
            & \dynkin[rootradius = 0.06cm, labels = {1,2,3, , ,$2n$}]D{xo*.***}
            \Tstrut\Bstrut
        \\
        \hline
            $\so_{4n}(\C)$ 
            & $\sl_{2n}(\C)$
            & $\so^*(4n)$
            & $\su^*(2n)$
            & \dynkin[rootradius = 0.06cm, labels = {1,2,3, , ,$2n$}]D{*o*.o*x}
            \Tstrut\Bstrut
        \\
        \hline
            $\so_{4n+2}(\C)$ 
            & $\so_{4n-2}(\C)$
            & \makecell[c]{$\so_{1,4n+1}(\R)$ \Tstrut\\ (quasi-compact)}
            & $\so_{4n-2}(\R)$
            & \dynkin[rootradius = 0.06cm, labels = {1,2,3, , ,$2n+1$}]D{x**.***}
            \Tstrut\Bstrut
        \\
        \hline
            $\so_{4n+2}(\C)$ 
            & $\sl_{2n+1}(\C)$
            & $\so_{2n+1, 2n+1}(\R)$
            & $\sl_{2n+1}(\R)$
            & \dynkin[rootradius = 0.06cm, labels = {1,2,3, , ,$2n+1$}]D{ooo.oox}
            \Tstrut\Bstrut
        \\
        \hline
            $\e_6(\C)$ 
            & $\so_{10}(\C)$
            & \makecell[c]{$\e_{6, -26}$\Tstrut \\ (quasi-compact)}
            & $\so_{1,9}(\R)$
            & \dynkin[rootradius = 0.06cm, labels = {1,2,3,4,5,6}]E{o****x}
            \Tstrut\Bstrut
        \\
        \hline
            $\e_7(\C)$ 
            & $\e_6(\C)$
            & $\e_{7, -25}$
            & $\e_{6,-26}$
            & \dynkin[rootradius = 0.06cm, labels = {1,2,3,4,5,6,7}]E{o****ox}
            \Tstrut\Bstrut
        \\
        \hline
    \end{tabular}
    \vspace{5pt}
    \caption{Real forms \emph{inner to a split real form} adapted to minuscule coweights $\lambda$, and their Satake diagrams with the (unpainted) node corresponding to $\lambda$ crossed out. If the real form $(\g^\vee)^\sigma$ is quasi-compact, this is remarked in the third column.}
    \label{table:split}
\end{table}

\begin{table}[ht]
    \centering
    \begin{tabular}{|c|c|c|c|c|}
        \hline
        $\g^\vee$ & $\el_\lambda'$ & $(\g^\vee)^\sigma$ & $(\el_\lambda^\sigma)'$ & Satake diagram 
        \Tstrut\Bstrut
        \\
        \hline
            $\sl_{2n}(\C)$ 
            & $\sl_{n}(\C)\oplus \sl_{n}(\C)$ 
            & $\su_{n,n}$ 
            & $\sl_n(\C)_\R$
            & \dynkin[rootradius = 0.06cm, labels = {1,2,3,$n$,,,$2n-1$}, fold]A{oo.oxo.oo}
            \Tstrut\Bstrut
        \\
        \hline
            $\so_{4n}(\C)$ 
            & $\so_{4n-2}(\C)$
            & \makecell[c]{$\so_{1, 4n-1}(\R)$\Tstrut \\ (quasi-compact)}
            & $\so_{4n-2}(\R)$
            & \dynkin[rootradius = 0.06cm, labels = {1,2,3, , ,$2n$}]D{x**.***}
            \Tstrut\Bstrut
        \\
        \hline
            $\so_{4n+2}(\C)$ 
            & $\so_{4n}(\C)$
            & $\so_{2, 4n}(\R)$
            & $\so_{1, 4n-1}(\R)$
            & \dynkin[rootradius = 0.06cm, labels = {1,2,3, , ,$2n+1$}]D{xo*.***}
            \Tstrut\Bstrut
        \\
        \hline
    \end{tabular}
    \vspace{5pt}
    \caption{Real forms \emph{not inner to a split real form} adapted to minuscule coweights $\lambda$, and their Satake diagrams with the (unpainted) node corresponding to $\lambda$ crossed out. If the real form $(\g^\vee)^\sigma$ is quasi-compact, this is remarked in the third column. A subscript $\R$ denotes that a complex Lie algebra is considered as its underlying real Lie algebra.}
    \label{table:non_split}
\end{table}

\clearpage
\printbibliography

\end{document}

%% file: preamble.tex
\usepackage[left=3cm,right=3cm,
    top=3cm,bottom=3cm,bindingoffset=0cm]{geometry}
\usepackage{graphicx}
\usepackage{hyperref}
\hypersetup{colorlinks,allcolors=black}

\usepackage{amsmath}
\usepackage{amsfonts}
\usepackage{amssymb}
\usepackage{amsthm}
\usepackage{mathtools}
\usepackage{tikz-cd}
\usepackage{fdsymbol}
\usepackage{dynkin-diagrams}

\usepackage{multicol}
\usepackage{makecell}

\usepackage{biblatex}
\newcommand\Tstrut{\rule{0pt}{2.6ex}}
\newcommand\Bstrut{\rule[-0.9ex]{0pt}{0pt}}

\newcommand{\ie}{i.\,e.}
\newcommand{\eg}{e.\,g.}
\newcommand{\cf}{cf.\,}

\newcommand{\Ad}{\operatorname{Ad}}
\newcommand{\ad}{\operatorname{ad}}
\newcommand{\Aut}{\operatorname{Aut}}
\renewcommand{\b}{\mathfrak{b}}
\newcommand{\B}{\mathcal{B}}
\newcommand{\C}{\mathbb{C}}
\newcommand{\Conj}{\operatorname{Conj}}
\newcommand{\Kir}{\mathcal{C}}

\newcommand{\D}{\mathcal{D}}
\newcommand{\e}{\mathfrak{e}}
\newcommand{\End}{\operatorname{End}}
\newcommand{\ev}{\operatorname{ev}}
\newcommand{\g}{\mathfrak{g}}
\newcommand{\GL}{\mathrm{GL}}
\newcommand{\gl}{\mathfrak{gl}}
\newcommand{\Gr}{\operatorname{Gr}}
\newcommand{\h}{\mathfrak{h}}

\newcommand{\Hom}{\operatorname{Hom}}

\newcommand{\id}{\operatorname{id}}
\newcommand{\Int}{\operatorname{Int}}
\newcommand{\N}{\mathbb{N}}
\newcommand{\kay}{\mathfrak{k}}
\newcommand{\el}{\mathfrak{l}}
\newcommand{\m}{\mathfrak{m}}
\renewcommand{\O}{\mathrm{O}}
\newcommand{\Out}{\operatorname{Out}}
\newcommand{\p}{\mathfrak{p}}
\newcommand{\PGL}{\mathrm{PGL}}
\newcommand{\pt}{\text{pt}}
\newcommand{\R}{\mathbb{R}}

\newcommand{\res}{\operatorname{res}}
\newcommand{\s}{\mathfrak{s}}

\newcommand{\SL}{\mathrm{SL}}
\renewcommand{\sl}{\mathfrak{sl}}
\newcommand{\so}{\mathfrak{so}}
\renewcommand{\sp}{\mathfrak{sp}}
\newcommand{\Spec}{\operatorname{Spec}}
\newcommand{\su}{\mathfrak{su}}
\newcommand{\te}{\mathfrak{t}}
\newcommand{\tr}{\operatorname{tr}}

\newcommand{\wt}{\operatorname{wt}}
\newcommand{\Z}{\mathbb{Z}}
\newcommand{\z}{\mathfrak{z}}
\newcommand{\inv}{^{-1}}

\theoremstyle{definition}
\newtheorem{defn}{Definition}[section]

\theoremstyle{plain}
\newtheorem{thm}[defn]{Theorem}
\newtheorem{prop}[defn]{Proposition}
\newtheorem{propdef}[defn]{Proposition and Definition}
\newtheorem{lem}[defn]{Lemma}
\newtheorem{cor}[defn]{Corollary}

\theoremstyle{remark}
\newtheorem{rmk}[defn]{Remark}
\newtheorem{expl}[defn]{Example}

\DeclareRobustCommand{\SkipTocEntry}[5]{}